\documentclass[a4paper]{amsart}
\usepackage[utf8]{inputenc}

\usepackage{amsmath,amssymb,mathtools,eqparbox}
\usepackage[left=3.5cm,right=3.5cm,top=2cm,bottom=2cm]{geometry}
\usepackage[pagebackref]{hyperref}
\usepackage[backgroundcolor=yellow]{todonotes}

\newtheorem{theorem}{Theorem}[section]
\newtheorem{lemma}[theorem]{Lemma}
\newtheorem{corollary}[theorem]{Corollary}
\newtheorem{proposition}[theorem]{Proposition}

\theoremstyle{definition}
\newtheorem{definition}[theorem]{Definition}

\theoremstyle{remark}
\newtheorem{remark}[theorem]{Remark}
\newtheorem{example}[theorem]{Example}

\numberwithin{equation}{section}

\newcommand*{\RR}{\mathbb{R}}
\newcommand*{\NN}{\mathbb{N}}

\DeclareMathOperator{\EE}{\mathbb{E}} 
\DeclareMathOperator{\PP}{\mathbb{P}} 
\DeclareMathOperator{\Ent}{Ent} 
\DeclareMathOperator{\Var}{Var} 
\DeclareMathOperator{\calE}{\mathcal{E}} 
\DeclareMathOperator{\calF}{\mathcal{F}} 
\DeclareMathOperator{\calG}{\mathcal{G}} 
\DeclareMathOperator{\calBn}{\mathcal{B}_n} 
\DeclareMathOperator{\calHd}{\mathcal{H}_d} 
\DeclareMathOperator{\supp}{supp}	
\DeclareMathOperator{\Unif}{Unif}	
\DeclareMathOperator{\argmin}{argmin}	
\DeclareMathOperator{\conv}{conv}	

\newcommand*{\ind}[1]{\mathbf{1}_{#1}}	
\newcommand*{\indbr}[1]{\ind{\{#1\}}}	
\newcommand{\Prob}[1]{\mathbb{P}\left( #1 \right)} 
\DeclarePairedDelimiter\norm{\Vert}{\Vert}
\DeclarePairedDelimiter\abs{\vert}{\vert}

\DeclarePairedDelimiter\ceil{\lceil}{\rceil}

\newcommand{\ub}[1]{^{(#1)}}
\def\ii{{\bf i}}

\title[Concentration for negatively dependent binary random variables]{Concentration inequalities for some negatively dependent binary random variables}

\author{Rados{\l}aw Adamczak}
\author{Bart{\l}omiej Polaczyk}
\address{Institute of Mathematics, University of Warsaw, Poland}
\email{R.Adamczak, B.Polaczyk @mimuw.edu.pl}

\date{Last changes: \today}
\thanks{Research partially supported by the National Science Centre, Poland, via the Sonata Bis grant no.\ 2015/18/E/ST1/00214 (RA) and the Preludium grant no.\ 2020/37/N/ST1/02667 (BP)}
\begin{document}

\begin{abstract}
We investigate concentration properties of functions of random vectors with values in the discrete cube, satisfying the stochastic covering property (SCP) or the strong Rayleigh property (SRP).

Our result for SCP measures include subgaussian inequalities of bounded-difference type extending  classical results by Pemantle and Peres and their counterparts for matrix-valued setting strengthening recent estimates by Aoun, Banna and Youssef. Under a stronger assumption of the SRP we obtain Bernstein-type inequalities for matrix-valued functions, generalizing recent bounds for linear combinations of positive definite matrices due to Kyng and Song.

We also treat in detail the special case of independent Bernoulli random variables conditioned on their sum for which we obtain strengthened estimates, deriving in particular modified log-Sobolev inequalities, Talagrand's convex distance inequality and, as corollaries, concentration results for convex functions and polynomials, as well as improved estimates for matrix-valued functions. These results generalize inequalities for the uniform measure on slices of the discrete cube, studied extensively by many authors. Our approach is based on recent results by Hermon and Salez and a general framework involving modified log-Sobolev inequalities on the discrete cube, which is of independent interest.

\medskip



\end{abstract}

\maketitle
\section{Introduction}
Investigating families of binary random variables with negatively dependent coordinates is an important problem from the point of view of computer science and combinatorics, which in the recent years has attracted considerable attention, see, e.g., \cite{MR1757964,MR1777538,MR2476782,MR3197973,MR3899605,garbe2018concentration,MR4010964,MR4003314,MR4108222,kathuria2020matrix}.
A wide and important class of such variables is constituted by those satisfying the \textit{strong Rayleigh property} (abbrev. SRP) introduced by Borcea et al.~\cite{MR2476782}.
More precisely, a probability measure $\pi$ on the hypercube $\calBn := \{0,1\}^n$ satisfies the SRP if its generating polynomial
\[
	\mathbb{C}^n \ni z \mapsto \sum_{x\in \calBn} \pi(x)\prod_{i=1}^n z_i^{x_i}
\]
has no roots $z$ whose all coordinates lie in the (strict) upper half-plane.
The examples of such measures are, e.g., the law of independent Bernoulli random variables conditioned on their sum, determinantal measures, uniform measure on the bases of balanced matroids, laws of point processes or measures obtained by running exclusion dynamics on the cube, cf. Pemantle and Peres~\cite{MR3197973}.

The main purpose of this article is to deepen the understanding of the concentration of measure phenomenon in the context of strong Rayleigh distributions and related classes of probability measures on the discrete cube.
In some of our considerations we will exploit only a more general notion of the \textit{stochastic covering property} (abbrev. SCP, cf.~Definition~\ref{d:scp}) introduced by Pemantle and Peres~\cite{MR3197973}, since this condition already turns out to provide a useful framework for proving concentration results~\cite{MR3197973,hermon2019modified,MR4108222,MR3899605,kathuria2020matrix,adamczak2020modified}. On the other hand for some more specialized inequalities we will restrict our attention to independent Bernoulli variables conditioned on their sum taking some fixed value. Distributions of this type generalize the uniform measure on slices of the discrete cube, related to the Bernoulli--Laplace model, which has been extensively studied, e.g., by Lee--Yau, Bobkov--Tetali, Gao--Quastel \cite{MR1675008,MR2283379, MR2023890} and more recently by Samson \cite{MR3690287} and Sambale--Sinulis \cite{sambale2020concentration}. The non-uniform distribution given by conditioned Bernoulli variables has found applications, e.g., in survey sampling being a  model of sampling without replacement from a finite population, with prescribed inclusion probabilities, which maximizes the entropy (often referred to as \emph{conditional Poisson sampling}). We refer to \cite{MR1488647, MR1790614, MR1311090,MR2225036,MR4010964} for properties and applications of this family of distributions.

\subsection{State of the art}
The landmark paper that initiated the study of concentration phenomenon implied by the SCP is due to Pematle and Peres~\cite{MR3197973} who, using the martingale method, proved a sub-Gaussian concentration bound for measures satisfying the SCP and functions that are Lipschitz with respect to the Hamming distance $d_H(x,y) = \sum_i \ind{x_i\neq y_i}$.
Recently, Hermon and Salez~\cite{hermon2019modified}, building on the works~\cite{MR1233852,MR2099650}, retrieved this estimate by proving that the SCP implies the modified log-Sobolev inequality.

These findings in terms of concentration of measure can be summarized as follows (we use the notation $\pi(f):=\int f\,d\pi$ for a probability measure $\pi$ on $\calBn$ and $f\colon \calBn\to\RR$).
\begin{theorem}[Pemantle--Peres~\cite{MR3197973}]\label{T:PP-main}
For a probability measure $\pi$ on $\calBn$ satisfying the SCP and any $f\colon\calBn\to\RR$ such that
\[
	\abs{f(x)-f(y)}
	\le
	d_H(x,y)
	\quad
	\forall\, x,y\in\calBn
\]
the following estimate holds for all $t>0$
\begin{equation}\label{eq:PP-main}
\pi\big(f > \pi(f) + t \big)
\le
\exp(-t^2/8n).
\end{equation}
If $\pi$ is $k$-homogeneous (i.e., it is supported on the set of binary vectors with exactly $k$ coefficients equal to one), then $n$ in the above expression can be replaced with $k$.
\end{theorem}

Recently, a sub-exponential version of Theorem \ref{T:PP-main} for matrix-valued functions has been shown by Aoun et al.~\cite{MR4108222}, who develop a general framework for deducing concentration bounds for matrix-valued functions from the Poincar\'{e} inequality.
A Bernstein-type bound for measures with the SRP, which in certain situations may give stronger concentration, has been also developed by Kyng and Song~\cite{MR3899605} for functions of the form $f(x) = \sum_{i=1}^n x_i C_i$, where $C_i$ are nonnegative definite matrices (see Theorem \ref{T:weak-martingale} and Remark \ref{re:K-S} below).

While concentration estimates and functional inequalities for general SCP measures are relatively recent, investigation of uniform measures on slices of the discrete cube in this context has much longer history. Such measures are of interest in relation to the Bernoulli--Laplace models of statistical physics and with uniform sampling without replacement. In particular Lee and Yau studied the Poincar\'e and log-Sobolev inequalities for such measures, whereas Bobkov--Tetali~\cite{MR2283379} and independently Gao--Quastel~\cite{MR2023890} investigated modified log-Sobolev inequalities relevant for concentration estimates. Strong concentration results for this case can be also obtained by projection from Talagrand's convex distance inequality for uniform measure on the symmetric group~\cite{MR1361756}. Samson~\cite{MR3690287} complemented this approach by proving corresponding transportation inequalities. Very recently Sambale and Sinulis~\cite{sambale2020concentration}, investigating general multislices, recovered convex distance inequalities by means of functional inequalities and also obtained concentration for polynomials. One should stress that concentration results for slices of the cube provided by the above references are substantially stronger than those coming from more general inequalities for SCP or SRP measures.

The uniform measure on slices of the cube can be seen as a special case of the distribution of independent Bernoulli random variables conditioned on their sum, when all the variables have the same probability of success. Such general distributions are known to be strong Rayleigh. To our best knowledge there has not been much work concerning refined concentration inequalities for general measures of this type. The only exception we are aware of is a recent article \cite{MR4010964} by Bertail and Cl\'{e}men\c{c}on in which the authors, motivated by applications to survey sampling, obtain precise Bernstein-type inequalities for linear functionals.

\subsection{Our contribution}
We develop two types of results.

Firstly, for general measures satisfying the SCP we extend the martingale argument from Pemantle and Peres~\cite{MR3197973} and generalize Theorem~\ref{T:PP-main} to Lipschitz functions with respect to more general weighted Hamming distances $d_\alpha(x,y) = \sum \alpha_i\ind{x_i\neq y_i}$ obtaining a bounded-difference type inequality.
We use the approach developed for the scalar case together with matrix bounded-difference inequality due to Tropp~\cite{MR2946459} to get an analogous concentration for matrix-valued functions, strengthening the results of Aoun et al.~\cite{MR4108222}, in particular obtaining a subgaussian inequality in place of a subexponential one.
Under a stronger assumption of the SRP we are also able to strengthen the Bernstein-type inequality of Kyng and Song and extend it from linear combinations with coefficients in nonnegative definite matrices to general functions satisfying a matrix bounded-difference type assumptions.

Secondly, building up on the work of Hermon and Salez~\cite{hermon2019modified}, we develop an abstract condition which implies Talagrand's convex distance inequality, matrix-Bernstein inequality and higher order concentration for tetrahedral polynomials.
Finally, we prove that this condition is satisfied for the distribution of Bernoulli random variables conditioned on their sum being equal to some constant, yielding all the aforementioned results in this case.

\subsection{Organization of the  article}
In Section~\ref{S:general_SCP_results} we present our results concerning concentration for general measures satisfying the SCP/SRP.
In Section~\ref{S:CB_results} we specialize our analysis to Bernoulli random variables conditioned on their sum being equal to some constant.
Then, in Section~\ref{S:abstract_formulations} we formulate an abstract framework that allows to deduce the results of Section~\ref{S:CB_results}.
Finally, all the proofs are presented in Sections~\ref{S:general_SCP_proofs},~\ref{S:abstract_proofs} and~\ref{S:CB_proofs}.

\section{Concentration under the SCP and SRP}\label{S:general_SCP_results}
In this section we present our concentration results for general measures satisfying the SCP or SRP.
Let us start with introducing some notation.
For $x=(x_1,\ldots,x_n)\in\calBn:=\{0,1\}^n$ and any $S \subset \{1,2,\ldots, n\}=: [n]$ we use the shorthand notation $x_S = (x_i)_{i\in S}$. For any $r\in [n]$ we denote $x_{>r} = (x_i)_{i>r}$ (and analogously with relations other than $>$)\footnote{We adopt the convention that if $x\in\calBn$ then $x_{>n}=\emptyset$ and as a consequence, e.g., $\Prob{\cdot\,\vert\,X_{>n} = \emptyset}=\Prob{\cdot}$.}.
We also write $x^i$ for the vector obtained from $x \in \calBn$ by flipping its i-th coordinate and $x^{ij}$ for the vector obtained by swapping the $i$-th and $j$-th coordinate, i.e., $x^{i}=x\pm e_i$ and if $x_i\neq x_j$ then $x^{ij}=x\pm e_i \mp e_j$ for $i,j\in [n]$, where $e_i\in\calBn$ is the vector with one on the $i$-th and zeros on the remaining coordinates; whereas $x^{ij} = x$ if $x_i=x_j$. We remark that the notation $x^{ij}$ should not be confused with $(x^i)^j$.
The law of a random variable~$X$ is denoted by $\mathcal{L}(X)$, whereas $\mathcal{L}(X|A)$ stands for the conditional law of $X$ given an event~$A$ (with an analogous convention for conditioning with respect to $\sigma$-fields or other random variables).

Below, we recall the definition of the SCP.
\begin{definition}[Stochastic covering property]\label{d:scp}
For $x,y\in\calBn$, we say that $x$ covers $y$, denoted $x \triangleright y$, if $x=y$ or $x=y+e_i$ for some $i\in [n]$.

A random variable $X$ taking values in $\calBn$ satisfies the SCP if for any $S\subset [n]$ and any $x,y\in\calBn$ satisfying $x_S\triangleright y_S$ there exists a coupling $(U,V)$ between the conditional distributions $\mathcal{L}(X_{S^c} \,\vert\, X_S = y_S)$ and $\mathcal{L}(X_{S^c} \,\vert\, X_S = x_S)$ such that $U \triangleright V$.
A measure $\pi$ satisfies the SCP if $X$ with law $\pi$ does so.
\end{definition}

\begin{remark}
As indicated in the introduction, the SCP is implied by the SRP, cf.~\cite{MR3197973}.
The opposite however is not true, as is demonstrated, e.g., by Cryan et al. in~\cite[Appendix~A]{MR4203344}, where the authors show that yet another possible generalization of the SRP, the strong log-concavity, is incomparable with the SCP.
In particular, they construct a distribution that is supported on the bases of a matroid, and that satisfies the SCP and violates the log-concavity (and whence the SRP as well).
\end{remark}

For a finite sequence $x$, we denote by $x^\downarrow$ the non-increasing rearrangement of the elements of $x$ and for $\alpha\in[0,\infty)^n =: \RR_+^n$ and $x,y\in\calBn$ we define the $\alpha$-weighted Hamming distance $d_\alpha(x,y)=\sum_i \alpha_i\ind{x_i\neq y_i}$.
Finally, for $p\in [1,\infty]$, $\abs{\cdot}_p$ is the $\ell_p$ norm on $\RR^n$ and $\abs{\cdot} := \abs{\cdot}_2$ denotes the Euclidean norm.

\medskip

The first main result of this paper is the following generalization of Theorem~\ref{T:PP-main}.
\begin{theorem}\label{T:d_alpha}
For a probability measure $\pi$ on $\calBn$ satisfying the SCP, any $f\colon \calBn\to\RR$ and $\alpha\in \RR_+^n$ such that
\[
	\abs{f(x)-f(y)}
	\le
	d_\alpha(x,y)
	\quad
	\forall\, x,y\in\calBn
\]
the following estimate holds for all $t>0$:
\begin{align*}
 \pi\big(f > \pi(f) + t \big) \le \exp(-t^2/8\abs{\alpha}^2).
\end{align*}
If $\pi$ is $k$-homogeneous then $8\abs{\alpha}^2$ in the above estimate can be replaced with $16\sum_{i=1}^k (\alpha_i^{\downarrow})^2$.
\end{theorem}
\begin{remark}
Theorem~\ref{T:d_alpha} implies Theorem~\ref{T:PP-main} (up to an absolute constant in the exponent) by taking $\alpha = (1,1,\ldots,1)$.
Moreover, by considering functions of the form $f(x) = \sum_i c_ix_i$ with $\abs{c}^2 \ll n\abs{c}^2_\infty$	in the non-homogeneous or $\sum_{i=1}^k (c_i^\downarrow)^2 \ll k\abs{c}^2_{\infty}$ in the $k$-homogeneous case, one can see that Theorem~\ref{T:d_alpha} can give substantially better concentration estimates than Theorem~\ref{T:PP-main}.
\end{remark}

We now formulate the matrix analogue of Theorem~\ref{T:d_alpha}.
To this end, let us denote the space of $d$-dimensional Hermitian matrices by $\calHd$, the identity matrix in $\calHd$ by $I_d$, the maximal eigenvalue of $H\in\calHd$ by $\lambda_{max}(H)$ and the operator norm of $H$ by $\norm{H}$.
\begin{theorem}\label{T:d_alpha_matrix}
For a probability measure $\pi$ on $\calBn$ satisfying the SCP, any $f\colon\calBn\to\calHd$ and $\alpha\in\RR^n_+$ such that
\begin{align}\label{eq:matrix-Lipschitz}
	\norm{f(x)-f(y)}
	\le
	d_\alpha(x,y)
	\quad
	\forall\, x,y\in\calBn
\end{align}
the following estimate holds for all $t>0$
\begin{align*}
 \pi\big( \lambda_{max}(f - \pi(f)) > t \big) \le
 d\exp(-t^2/32\abs{\alpha}^2).
\end{align*}
If $\pi$ is $k$-homogeneous then $32\abs{\alpha}^2$ in the above estimate can be replaced with $64\sum_{i=1}^k (\alpha_i^{\downarrow})^2$.
\end{theorem}
\begin{remark}
Recently, Aoun~et~al.~\cite{MR4108222} showed that for any $k$-homogeneous probability measure $\pi$ on $\calBn$ satisfying the SCP and any $f\colon \calBn \to \calHd$ such that
\[
	\norm{f(x)-f(y)}
	\le
	d_H(x,y)
	\quad
	\forall\, x,y\in\calBn
\]
the following estimate applies
\begin{equation}\label{eq:Aoun_SCP}
	\pi\big(
	\lambda_{max}(f - \pi(f)) > t
	\big)
	\le
	d\exp\Big(
	-\frac{t^2}{8k+2t\sqrt{2k}}
	\Big).
\end{equation}
The exponent in~\eqref{eq:Aoun_SCP} is proportional to $-t/2\sqrt{2k}$ for $t$ big enough and whence for such $t$ Theorem~\ref{T:d_alpha_matrix} applied with $\alpha=(1,\ldots,1)$ strengthens on~\eqref{eq:Aoun_SCP} (and on an analogous result from~\cite{kathuria2020matrix}) as it yields a sub-Gaussian estimate.
\end{remark}
\begin{remark}
Using semigroup techniques together with matrix concentration results implied by the Poincar\'{e} inequality due to Aoun et al.~\cite{MR4108222}, we are also able to derive a sub-exponential concentration inequality for general measures satisfying the SCP under weaker assumptions on $f$ than those of Theorem \ref{T:d_alpha_matrix}, cf.~Remark~\ref{R:matrix_bernstein}.
\end{remark}

When comparing the inequality of Theorem \ref{T:d_alpha_matrix} or the results from~\cite{MR4108222} with results for matrix-valued functions of independent random variables, one can ask if it is possible to weaken the assumptions on the function $f$ and instead of the Lipschitz constant with respect to $d_\alpha$ use some weaker parameter, involving bounds on the increments of the function in terms of the positive semidefinite order. In many situations one encounters functions for which $(f(x) - f(x^i))^2 \preccurlyeq C_i^2$ where $C_i$ are some positive semidefinite matrices and $\preccurlyeq$ stands for the positive semidefinite order (note that considering arbitrary matrices $C_i$ is a generalization of the condition  \eqref{eq:matrix-Lipschitz}, which corresponds to the special case $C_i^2= \alpha_i^2 I_d$). The simplest, yet important situation of this type is given by $f(x) =  \sum_{i=1}^n x_iC_i$. Inequalities for such functions together with algorithmic applications were considered by Kyng and Song in \cite{MR3899605}. It turns out that their approach can be adapted to the setting of general functions, yielding the following theorem.

\begin{theorem}\label{T:weak-martingale}
Let $\pi$ be a $k$-homogeneous probability measure $\calBn$ satisfying the strong Rayleigh property and $f\colon\calBn\to\calHd$ be such that there exists a sequence $C_1,\ldots,C_n\in\calHd$ satisfying
\begin{align}\label{eq:matrix-bounded difference-assumption}
(f(x) - f(x^i))^2 \preccurlyeq C_i^2
\quad
\forall\, x\in\calBn,\, i\in [n].
\end{align}
Then for any $t > 0$,
\begin{equation}\label{eq:matrix-bounded-diff-result}
   \pi\big( \lambda_{max}(f - \pi(f)) > t \big) \le d\exp\Big(-\frac{t^2}{8
    \|\pi (\tilde{f}) \|\log (ek) + \frac{4}{3}Kt}\Big),
\end{equation}
where $\tilde{f}(x)=\sum_{i=1}^n x_i C^2_i$ and $K = \max_{i\le n} \|C_i\|$.
\end{theorem}
\begin{remark}
In fact, the only place in the proof of Theorem~\ref{T:weak-martingale} where we use the SRP in its full strength is to get that $\PP(X_i=1 \,\vert\, X_{i_1}=1,\ldots,X_{i_l}=1) \le \PP(X_i=1)$ for $X\sim\pi$ and any $i,k\in[n]$ and $i_1,\ldots, i_k\subset[n]\setminus\{i\}$.
Therefore, in Theorem~\ref{T:weak-martingale} it suffices to assume that $\pi$ satisfies the SCP and negative association, which is implied by the SRP, cf.~\cite{MR3197973}.
\end{remark}
\begin{remark}\label{re:K-S}
It is natural to expect that $\log(ek)$ in \eqref{eq:matrix-bounded-diff-result} is just an artefact of the proof. Removing it even just for functions of the form $f(x) = \sum_{i=1}^n x_i C_i$ for positive semidefinite matrices $C_i$ would lead to improvement of certain algorithmic constructions related to graph sparsifiers obtained via random spanning trees, cf. \cite{MR3899605}.

Let us also point out that even though Theorem \ref{T:weak-martingale} applies to general functions, when specialized to the function $f$ as discussed above, it improves slightly on the results from~\cite{MR3899605}, which instead of $\|\pi(\tilde{f})\|$ use a larger quantity $K\|\pi (\hat{f})\|$ with $
\hat{f}(x) = \sum_{i=1}^n x_i C_i$ (recall that $C_i$'s are nonnegative definite). One should stress however that we rely on the approach worked out in \cite{MR3899605}.
\end{remark}

\section{Concentration for conditional Bernoullis}\label{S:CB_results}
In this section, we present our concentration results concerning Bernoulli random variables conditioned on their sum being constant.
These include Talagrand's convex distance inequality, matrix-Bernstein inequality and concentration for polynomials.

We start with introducing the notation.
For a sequence $p=(p_1,\ldots,p_n)\in (0,1)^n$, let $B=(B_1,\ldots,B_n)$ be a sequence of independent Bernoulli random variables with probabilities of success $p_i$, i.e., $\PP(B_i=1)=1-\PP(B_i=0)=p_i$ for $i\in [n]$.
Finally, set $X=(X_1,\ldots,X_n) \sim \mathcal{L}\big( B\,\vert \sum_i B_i = k \big)$ for some $k\in\{0,\ldots,n\}$ and denote the distribution of $X$ by $\pi(p,k)$.

Our first contribution is a counterpart of the celebrated convex distance inequality, introduced for the first time by Talagrand~\cite{MR964871} for product measures on the cube.
\begin{theorem}\label{T:SCP-Talagrand}
If $\pi\sim \pi(p,k)$ for some $p\in (0,1)^n$ and $k\in\{0,\ldots,n\}$, then for any $A\subset \calBn$,
\[	
 \pi(A)\pi\big(d_T^2(\cdot, A)/84 \big) \le 1,
\]
where
\[
	d_T(x,A) = \sup_{\alpha\colon \abs{\alpha}\le 1} d_\alpha(x,A)
	\quad
	\text{for}
	\quad
	x\in\calBn,\,
	A\subset\calBn.	
\]
\end{theorem}
Let $\mathbb{M}_\pi f$ denote any median of $f$ with respect to the measure $\pi$.
A classical consequence of Theorem~\ref{T:SCP-Talagrand} is the following fact regarding the concentration around the median of {convex} functions~\cite{MR3185193}. Let us recall the classical observation that subgaussian concentration around median and mean for convex Lipschitz functions are equivalent up to the change of constants by a universal factor.

\begin{corollary}\label{C:Talagrand_concentration}
If $\pi\sim \pi(p,k)$ for some $p\in (0,1)^n$ and $k\in\{0,\ldots,n\}$, then for any convex function $f\colon\RR^n\to\RR$ that is $L$-Lipschitz with respect to the standard Euclidean distance on $\RR^n$ and any $t>0$,
\[
 \pi\big( \abs{f-\mathbb{M}_\pi f}  > t\big) \le
 4\exp\big( -t^2/84L^2 \big).
\]
\end{corollary}
\begin{remark}
If one is interested just in the lower tail of a convex function, then one can in fact replace the Lipschitz constant $L$ by $\pi \big(|\nabla f|\big)$ or even certain quantiles of $|\nabla f|$. We do not pursue this direction here and refer the reader to \cite{MR3892322}.
\end{remark}

\begin{remark}
If $f\colon \calBn \to \RR$ is $d_{\alpha}$ 1-Lipschitz, then it can be extended to a function on $\RR^n$ which is $\abs{\alpha}$-Lipschitz with respect to the standard Euclidean distance.
Therefore, Corollary~\ref{C:Talagrand_concentration} counterparts Theorem~\ref{T:d_alpha} in the sense that it yields the same concentration profile while allowing for a weaker Lipschitz condition on $f$ at the cost of assuming convexity.
\end{remark}

Our next result concerns concentration for matrix-valued functions under weaker assumptions than those in Theorem~\ref{T:d_alpha_matrix}.
\begin{theorem}\label{T:Bernoullis_matrix}
Let $\pi\sim \pi(p,k)$ for some $p\in (0,1)^n$ and $k\in\{0,\ldots,n\}$.
Assume that $f\colon\calBn\to\calHd$ is such that there is a sequence of positive semidefinite matrices $C_1,\ldots,C_n$ satisfying
\begin{equation}\label{eq:matrix_condition}
(f(x) - f(x^i))^2
\preccurlyeq
C_i^2,
\quad\forall\;x\in\calBn,\,i\in[n],
\end{equation}
where $\preccurlyeq$ denotes the partial ordering of the set of positive semidefinite matrices.
Define the variance proxy
\begin{equation*}
 \sigma^2 =
 16\sup\Big\{\,
 \big\|
 	\sum_{i\in \mathcal{I}} C_i^2
 \big\|
 \,\colon\,
 \abs{\mathcal{I}} = k,\,
 \mathcal{I}\subset [n]
 \,\Big\}.
\end{equation*}
Then for any $t>0$,
\[
 \pi\big(
 \lambda_{max}(f - \mu(f)) > t
 \big)
 \le
 d\exp\big(
 -t^2/(\sigma^2+{\sigma}t)
 \big).
\]
\end{theorem}
\begin{remark}
Condition~\eqref{eq:matrix_condition} implies that $f$ is 1-Lipschitz with respect to the distance $d_\alpha$ with $\alpha_i={\norm{C_i}}$.
On the other hand, for many choices of matrices $C_1,\ldots,C_n$ it happens that $\sigma^2 \ll  \sum_{i=1}^k\big(\norm{C_i}^2\big)^{\downarrow}$ as $n,k \to \infty$.
Therefore, while yielding only sub-exponential concentration as opposed to the sub-Gaussian one given by Theorem~\ref{T:d_alpha_matrix}, Theorem~\ref{T:Bernoullis_matrix} may improve significantly on Theorem~\ref{T:d_alpha_matrix} through better parameters in the exponent.
\end{remark}
\begin{remark}\label{R:matrix_bernstein}
By an adaptation of the proof of Theorem~\ref{T:Bernoullis_matrix}, one can obtain a similar result for general $k$-homogeneous measures satisfying the SCP condition with the variance proxy parameter
\[
	\sigma^2=
	8\sup\Big\{\,
		\big\|
		{\sum_{i\in\mathcal{I}}C_i^2}
		\big\|
		+
		k
		\max_{i\notin\mathcal{I}}
		\big\|{C_i^2}\big\|
		\colon
		\abs{\mathcal{I}}\le k,\, \mathcal{I}\subset [n]
	\,\Big\}.
\]
\end{remark}

Finally, let us turn to the higher order concentration.
By the Fourier--Walsh expansion (see e.g., \cite{MR3443800}), every function $f\colon \calBn \to \RR$ can be written in a unique way as a tetrahedral polynomial, i.e., a polynomial which is affine with respect to every variable (in particular the degree of the polynomial is at most $n$).
Therefore in what follows we restrict our attention to this representation.
In particular, when we speak about the gradient $\nabla f = (\partial_1 f,\ldots,\partial_n f)$ or higher order derivatives $\nabla^k f$, we always think of the usual derivatives of the polynomial function on $\RR^n$ given by the tetrahedral representation of $f$ (sometimes referred to as the harmonic extension of $f$). We remark that the directional derivatives $\partial_i f$ coincide on $\calBn$ with the discrete derivatives of $f$ given by $D_i f(x) = f(\max(x,x^i)) - f(\min(x,x^i))$, where the maximum and minimum are taken coordinatewise.

In order to formulate concentration of measure estimates for tetrahedral polynomials, we need to introduce a family of injective tensor product norms on $d$-index matrices ($d$-tensors).
Let us recall the notation introduced by Latała in~\cite{MR2294983}.

Let $|I|$ be the cardinality of a set $I$ and for $\ii = (i_1,\ldots,i_d) \in [n]^d$ let $|\ii| = \max_{j\le d} {i_j}$ and $|\ii_I| = \max_{j \in I} i_j$.
Denote by $P_d$ the set of partitions of $[d]$ into nonempty, pairwise disjoint sets. For a partition $\mathcal{I} =\{I_1,\ldots,I_k\} \in P_d$, and a $d$-indexed matrix $A = (a_\ii)_{\ii \in [n]^d}$, define
\begin{align*}
\|A\|_{\mathcal{I}}=\sup\Big\{\sum_{\ii\in [n]^d} a_{\ii}\prod_{l=1}^k x\ub{l}_{\ii_{I_l}}\colon
\abs{(x\ub{l}_{\ii_{I_l}})}\leq 1, 1\leq l\leq k \Big\},
\end{align*}
where $\abs{(x_{\ii_{I_l}})} = \sqrt{\sum_{|\ii_{I_l}|\le n} x_{\ii_{I_l}}^2}$.
Therefore, for example,
\begin{align*}
\|(a_{ij})_{i,j\le n}\|_{\{1,2\}}&= \sup\Big\{ \sum_{i,j\le n} a_{ij}x_{ij}\colon \sum_{i,j\le n} x_{ij}^2 \le 1\Big\} = \sqrt{\sum_{i,j\le n}a_{ij}^2} = \|(a_{ij})_{i,j\le n}\|_{HS},\\
\|(a_{ij})_{i,j\le n}\|_{\{1\}\{2\}}&= \sup\Big\{ \sum_{i,j\le n} a_{ij}x_iy_j\colon \sum_{i\le n} x_{i}^2\le 1,\sum_{j\le n}y_j^2 \le 1\Big\} = \|(a_{ij})_{i,j\le n}\|,\\
\|(a_{ijk})_{i,j,k\le n}\|_{\{1,2\} \{3\}} &= \sup\Big\{ \sum_{i,j,k\le n} a_{ijk}x_{ij}y_k\colon \sum_{i,j\le n} x_{ij}^2\le 1,\sum_{k\le n}y_k^2 \le 1\Big\},
\end{align*}
where $\|\cdot\|_{HS}$ and $\|\cdot\|$ denote the Hilbert--Schmidt and the operator norm respectively.
\begin{theorem}\label{T:polynomials}
If $\pi\sim \pi(p,k)$ for some $p\in (0,1)^n$ and $k\in\{0,\ldots,n\}$, then for any tetrahedral polynomial $f\colon\calBn\to\RR$ of degree $d$,
\[
\pi\big(\big|f - \pi(f)\big| \ge t\big)
\le
2\exp\Big(
-\frac{1}{C_d}\min_{1\le r \le d} \min_{\mathcal{J}\in P_r}
\Big(
\frac{t}{\|\pi( \nabla^r f)\|_{\mathcal{J}}}
\Big)^{2/|\mathcal{J}|}
\Big),
\]
where $C_d$ is a constant depending only on the degree $d$ of $f$.
\end{theorem}

Inequalities of this type for polynomials of arbitrary degree were introduced for the first time by Latała \cite{MR2294983} for tetrahedral polynomials in i.i.d. standard Gaussian variables. Subsequently they were extended to general polynomials in independent subgaussian random variables and to certain dependent situations related to Glauber dynamics (see \cite{MR3383337,MR3949267,MR4091094,sambale2020concentration,adamczak2020modified}). We remark that in the independent, subgaussian case and $d=2$ they reduce to the well known Hanson--Wright inequality for quadratic forms, which has proved useful in non-asymptotic analysis of random matrices and in asymptotic geometric analysis (see, e.g., \cite[Chapter 6]{MR3837109}). It is worth mentioning that in the Gaussian case they may be reversed up to the value of the absolute constants, thus Theorem 3.8 shows that the measures $\pi(p,k)$ exhibit Gaussian type concentration for polynomials.
While calculating the norms $\|\cdot\|_\mathcal{J}$ is usually difficult, estimating them is sometimes possible, leading to applications involving subgraph counts (in the Erd\H{o}s-R\'enyi case or for some models of random graphs with dependencies \cite{MR3383337,MR4091094,sambale2020concentration}) or to statistical applications, e.g., in testing Ising models \cite{10.1145/3406325.3451074} and signal processing \cite{MR3788186}.

\section{Abstract formulations}\label{S:abstract_formulations}
In this section we recall some notions from the theory of Markov semigroups and formulate the abstract counterparts of the results of Section~\ref{S:CB_results} and of Theorem~\ref{T:d_alpha}.
We believe that the results presented in this section might be of separate interest as they provide a general framework for proving concentration on the hypercube.
We stress that most of the proof techniques that we exploit were known previously -- our main contribution is the abstract formulation of these results by means of the novel stability condition (cf. Definition~\ref{D:stability}).

Throughout this section we will rely on the usual notions from the theory of Markov processes and Dirichlet forms specialized to finite state space. We will briefly recall them and refer to \cite{MR2574430,MR3075390,MR2283379} for details.

\subsection{Modified log-Sobolev inequalities}
Let $L$ be the generator of a jump Markov process on some finite probability space $(M,\pi)$. In what follows we will sometimes treat $L$ as a linear operator on $\RR^M$ and sometimes identify it with the corresponding matrix, indexed by the elements of $M$.

Assume that $L$ satisfies the detailed-balance condition
\begin{equation}\label{eq:detailed-balance}
	\forall\; x,y\in M\quad
	\pi(x)L(x,y) = \pi(y)L(y,x),
\end{equation}
which implies that $\pi$ is a stationary measure for the Markov process and $L$ is self-adjoint on $L^2(\pi)$.
In this article we consider only Markov processes satisfying the above condition, which may not be stated explicitly in all the results.

For a given $L$, we define $\Delta(L) := \max_x -L(x,x)=\max_x \sum_{y\colon y\neq x} L(x,y)$ and write $\calE(f,g) = -\pi(fLg)$ for the Dirichlet form associated with~$L$.
In particular $\calE(f,g) = \pi\big( \Gamma(f,g) \big)$, where $\Gamma\colon  \RR^M \times \RR^M \to \RR^M$ given by
\begin{equation}\label{eq:Gamma_def}
 \Gamma(f,g)(x)
 =
 \frac{1}{2}
 \sum_{y\in M} (f(x)-f(y))(g(x)-g(y))L(x,y)
\end{equation}
is the corresponding carr\'e-du-champ operator.
We use shorthand notation $\Gamma(f,f)=:\Gamma(f)$ and observe that by the detailed-balance condition~\eqref{eq:detailed-balance} we have $\pi\big( \Gamma(f) \big)=\pi(\Gamma_+(f))$, where
\begin{equation}\label{eq:Gamma+def}
	\Gamma_+(f)(x)
	=
	\sum_{y\in M} (f(x)-f(y))_+^2L(x,y).
\end{equation}

Finally, we denote by $\rho(L)$ the best (the greatest) constant such that the following modified log-Sobolev inequality is satisfied
\begin{equation}\label{eq:mLSI}
\rho(L)\Ent_\pi(f) \le \calE(f,\log f)
\end{equation}
for all non-constant functions $f\colon M\to [0,\infty)$, where $\Ent_\pi(f) = \pi(f\log f)-\pi(f)\log\pi(f)$ is the entropy functional (we adopt the convention $0\log 0 = 0$).
We remark that $\rho(L)$ is positive iff $L$ is irreducible on the support of $\pi$ (see the discussion in~\cite{MR2283379} and~\cite[Chapter~12]{MR3726904}).
In what follows we will restrict our attention to this situation, without mentioning this assumption explicitly in each statement.

A classical observation, often referred to as Herbst's argument (cf. the monographs~\cite{MR1849347} by Ledoux and~\cite{MR3185193} by Boucheron et al.), says that for any $f\colon M\to\RR$,
\begin{equation}\label{eq:Herbsst_arg}
 \pi\big(f > \pi( f )+ t\big) \le \exp(-t^2\rho(L)/4\norm{\Gamma_+(f)}_\infty),
\end{equation}
where $\|\cdot\|_\infty$ stands for the norm in $L^\infty(\pi)$.
\subsection{Flip-swap random walks}
After Hermon and Salez~\cite{hermon2019modified}, we say that a kernel $L$ generates a flip-swap random walk if $L(x,y)>0$ implies that $x=y^i$ for some $i\in [n]$ (i.e., $x$ and $y$ differ by a flip) or $x=y^{ij}$ for some $i\neq j$, $i,j\in [n]$ (i.e., $x$ and $y$ differ by a swap).
The main contribution of~\cite{hermon2019modified} can be stated in the following way.
\begin{theorem}[Hermon--Salez~\cite{hermon2019modified}]\label{T:HS-main}
For any measure $\pi$ on $\calBn$ satisfying the SCP, there exists a kernel $L$ generating a reversible flip-swap random walk with stationary measure $\pi$ such that $\rho(L)\ge 1$ and $\Delta(L)\le n$.
If $\pi$ is also $k$-homogeneous, then $\Delta(L)\le 2k$ as well.
\end{theorem}

Theorem~\ref{T:HS-main}, by means of Herbst's argument~\eqref{eq:Herbsst_arg}, implies (up to an absolute constant in the exponent) the estimate from Theorem~\ref{T:PP-main} after observing that for a flip-swap random walk and any $f\colon\calBn\to\RR$ that is 1-Lipschitz with respect to the Hamming distance $d_H$
\begin{align}\label{eq:vf-dh-estimate}
 \norm{\Gamma_+(f)}_\infty \le
 \Delta(L) \cdot\max_{x,y\in\calBn} \{\,(f(y)-f(x))_+^2 \colon L(x,y) > 0\,\} \le
 4\Delta(L).
\end{align}

Finally, Theorem~\ref{T:HS-main} is constructive -- in Section~\ref{S:CB_proofs} we revisit the inductive construction of $L$ from~\cite{hermon2019modified} in the context of conditional Bernoulli distribution.

\begin{remark}
There are many examples of flip-swap random walks on the hypercube in the literature, including, e.g., the Bernoulli--Laplace model, Glauber dynamics or base exchange random walk on matroids, cf. e.g.,~\cite{MR2283379,MR2094147,MR4091094,MR4203344}.
We note that the results of this section apply to any flip-swap random walk as long as we have control of its stability (cf. Definition~\ref{D:stability}) constant.
\end{remark}

It turns out that for the proofs of all the statements of Section~\ref{S:CB_results} it suffices to demonstrate that the following condition is true for some reversible generator $L$ with stationary measure $\pi(p,k)$ for which the modified log-Sobolev inequality~\eqref{eq:mLSI} is known.
\begin{definition}[Stability condition]\label{D:stability}
Let $L$ be a generator of a flip-swap random walk on $\calBn$ with invariant probability distribution $\pi$.
We say that the pair $(L,\pi)$ meets the stability condition with constant $R\ge 0$ (i.e., is $R$-stable) if it satisfies the modified log-Sobolev inequality~\eqref{eq:mLSI} and
\begin{align}\label{eq:stability-condition}
\max_{x\in\supp\pi;\,i\in [n]} \sum_{y\colon y_i\neq x_i} L(x,y) \le R\rho(L).
\end{align}
If it is clear from the context which measure $\pi$ is associated with $L$, we will often omit it in the discussion and simply say that $L$ is $R$-stable.
\end{definition}

\begin{remark}\label{rem:R-lower-bound}
If $\pi$ is not concentrated on a single point, then a random walk on $\calBn$ with a generator $L$ that satisfies the modified log-Sobolev inequality~\eqref{eq:mLSI} may be at best $0.25$-stable (i.e., $R \ge 0.25$).
Indeed, in this case there exists $i$ such that $\pi(\{x_i=1\}), \pi(\{x_i=0\})>0$.
If $L$ satisfies the modified log-Sobolev inequality, then it also satisfies the Poincar\'e inequality $\frac{1}{2}\rho(L)\Var_\pi(f) \le \calE(f,f)$, see e.g., \cite[Proposition B.5]{adamczak2020modified}.
Therefore, by the stability condition~\eqref{eq:stability-condition} applied to the function $f(x) = \indbr{x_i=1}$ and reversibility of~$L$ we get that
\begin{align*}
	R\rho(L)\pi(\{x_i=1\})
	&\ge
	\sum_{x\colon x_i = 1}
	\sum_{y\colon y_i = 0}L(x,y)\pi(x)
	\\&=
	\sum_{x,y}
	(x_i-y_i)^2_+L(x,y)\pi(x)
	\\&=
	\calE(f, f)
	\\&\ge
	\frac{1}{2}\rho(L)\Var_\pi (f)
	=
	\frac{1}{2}\rho(L)\pi(\{x_i=1\})\pi(\{x_i=0\}),
\end{align*}
which gives $R \ge 0.5 \cdot \pi(\{ x_i=0 \})$.
Similarly, by considering $f(x) = \indbr{x_i=0}$ we get that $R \ge 0.5\cdot \pi(\{ x_i=1 \})$ as well, yielding $R \ge 0.25$.

This bound is optimal, as can be seen for $\pi$ being the uniform measure on $\calBn$ and $L(x,y) = 1$ if there exists $i$ such that $y= x^i$, $L(x,y) = -n$ if $y=x$ and $L(x,y) = 0$ otherwise (this corresponds to the special case of Glauber dynamics, in which at rate $n$, a random coordinate is flipped). In this case $\rho(L) = 4$ (see \cite[Example 3.7]{MR2283379}, note a different normalization of both the Dirichlet form and the constant in the modified log-Sobolev inequality), whereas for all $x\in\calBn$
\begin{displaymath}
\max_{i}\sum_{y\colon y_i\neq x_i} L(x,y) = L(x,x^i) = 1 = 0.25 \cdot \rho(L).
\end{displaymath}
\end{remark}

Let us illustrate the notion of $R$-stability with another classical example.
\begin{example}[Bernoulli--Laplace model]
Let $\pi$ be the uniform measure on the slice of $\calBn$ consisting of elements with exactly $k$ ones and let $L$ be given by $Lf(x) = \frac{1}{n}\sum_{i< j} (f(x^{ij}) - f(x))$ (thus the corresponding Markov process at rate $(n-1)/2$ swaps a uniformly chosen pair of coordinates). In the matrix form this corresponds to $L(x,y) = \frac{1}{n}$ if $x \neq y$ and $y = x^{ij}$, $L(x,x) = - k(n-k)/n$ and $L(x,y) = 0$ otherwise. It has been proved in \cite{MR2023890} and independently in \cite{MR2283379} that $\rho_0(L) \ge 1/2$. At the same time $\sum_{y\colon y_i \neq x_i} L(x,y)$ equals to $(n-k)/n$ if $x_i = 1$ and to $k/n$ otherwise. This shows that $L$ is $2$-stable, independently of $n$ and $k$. As mentioned in the introduction, the uniform measure on the slice of the discrete cube can be interpreted as the distribution of i.i.d. Bernoulli variables conditioned on their sum being equal to $k$. In Proposition~\ref{P:L-stability} we generalize the above observation on stability and show that if $\mu$ is the law of general independent Bernoulli variables conditioned on their sum being equal to a fixed constant, there exists a $2$-stable generator of a random walk reversible with respect to $\mu$.
\end{example}

\begin{remark}
Observe that the notion of stability is invariant under scaling of $L$ (change of time), i.e., if $L$ is $R$-stable then so is $aL$ for any $a > 0$.
This leads to a tensorization property for measures admitting an $R$-stable generator. More precisely, let $\pi_1,\ldots,\pi_m$ be measures on $\mathcal{B}_{n_1},\ldots,\mathcal{B}_{n_m}$, for which there exist reversible flip-swap random walks with $R$-stable generators $L_1,\ldots,L_m$.
By changing time, we can assume without loss of generality that $\rho(L_i) = \rho$ for all $i \le m$. Let $n = n_1+\ldots+n_m$ and consider the product measure $\pi = \pi_1\otimes\cdots\otimes \pi_m$ on $\calBn$ together with the generator $L = L_1+\ldots+L_m$, where we think of $L_i$ as acting only on the $i$-th block of coordinates on $\calBn = \mathcal{B}_{n_1}\times\cdots\times \mathcal{B}_{n_m}$, i.e., we identify $L_i$ with its tensor product with identity on $\otimes_{j\neq i} \RR^{\mathcal{B}_{n_j}}$. In the matrix form we have the representation
\begin{displaymath}
  L(x,y) = \sum_{i=1}^m L_i(P_i x,P_i y) \prod_{j\neq i} \indbr{P_j x = P_j y},
\end{displaymath}
where $P_j \colon \calBn \to \mathcal{B}_j$ is the projection onto the $j$-th factor in the product $\calBn = \mathcal{B}_{n_1}\times\cdots\times \mathcal{B}_{n_m}$
Thanks to the well known tensorization property of the entropy (see, e.g., \cite[Chapter 3]{MR1845806}) we have $\rho(L) = \rho$.
Moreover, for $i \in (n_1+\ldots+n_{j-1},n_1+\ldots+{n_{j}}]$,
\[
\sum_{y\in \calBn \colon y_i\neq x_i} L(x,y) = \sum_{y\in\mathcal{B}_{n_j}\colon y_l \neq (P_j x)_l} L_j((P_jx),y) \le R\rho,\] where $l = i - (n_1+\ldots+n_{j-1})$.
Thus $L$ is indeed $R$-stable.

This observation allows in particular to extend all the theorems od Section \ref{S:CB_results} to product of measures $\pi(n,k)$ allowing for more general conditioning of Bernoulli variables.

\end{remark}
\subsection{Abstract formulations}\label{ss:abstract-formulations}
Finally, let us present the counterparts of the results of Section~\ref{S:general_SCP_results} and of Theorem~\ref{T:d_alpha} from Section~\ref{S:CB_results} in the abstract language of the stability condition~\eqref{eq:stability-condition}.
We stress here that it is the sole property needed for their proofs, which are deferred to Section~\ref{S:CB_proofs}.

We start with a bounded-difference type inequality for real valued functions.

\begin{proposition}\label{P:abstract_dalpha}
If a flip-swap random walk on $\calBn$ with stationary distribution $\pi$ and generator $L$ satisfies the stability condition~\eqref{eq:stability-condition}, then for any $f\colon \calBn\to\RR$ and $\alpha\in \RR_+^n$ such that
\[
	\abs{f(x)-f(y)}
	\le
	d_\alpha(x,y)
	\quad
	\forall\, x,y\in\calBn
\]
the following estimate holds for all $t>0$
\begin{align*}
 \pi\big(f > \pi(f) + t \big) \le
 \exp\Big(
 -\frac{t^2}{8R \abs{\alpha}^2}
 \Big).
\end{align*}
In the above estimate one can also replace $8\abs{\alpha}^2$ with $16\sum_{i=1}^{\ceil{\Delta(L)/R\rho(L)}}(\alpha_i^\downarrow)^2$.
\end{proposition}

\begin{remark}
Using the definitions of $R$-stability and of $\Delta(L)$ one can  see that $\Delta(L)/R\rho(L) \le n$ and if $\pi$ is $k$-homogeneous, then
$\Delta(L)/R\rho(L) \le k$.
\end{remark}

Let us now pass to the matrix-valued case.

\begin{proposition}\label{P:abstract_matrix}
Let a flip-swap random walk on $\calBn$ with stationary distribution $\pi$ and generator $L$ satisfy the stability condition~\eqref{eq:stability-condition}.
Assume also that $f\colon\calBn\to\calHd$ is such that there is a sequence of positive semidefinite matrices $C_1,\ldots,C_n$ satisfying
\begin{equation}\label{eq:abstract_matrix_f_condition}
(f(x) - f(x^i))^2
\preccurlyeq
C_i^2
\quad\forall\;x\in\calBn,\,i\in[n],
\end{equation}
where $\preccurlyeq$ denotes the positive semidefinite order on the set of symmetric matrices.
Set the variance proxy
\[
 \sigma^2 =
 8R
 \cdot
 \sup\Big\{\,
 \big\|
 	\sum_{i\in \mathcal{I}} C_i^2
 \big\|
 \,\colon\,
 \abs{\mathcal{I}} = \ceil{\Delta(L)/R\rho(L)},\,
 \mathcal{I}\subset [n]
 \,\Big\}.
\]
Then for any $t>0$,
\[
 \pi\big(
 \lambda_{max}(f - \pi(f)) > t
 \big)
 \le
 d\exp\big(
 -t^2/(\sigma^2+\sigma t)
 \big).
\]
\end{proposition}
Our next proposition is the convex distance inequality under $R$-stability.

\begin{proposition}\label{P:abstract_Talagrand}
If a flip-swap random walk on $\calBn$ with some stationary distribution $\pi$ and a generator $L$ satisfies the stability condition~\eqref{eq:stability-condition}, then for any set $A\subset\calBn$
\[
\pi(A)
\pi\Big(
\exp\Big(\frac{1}{40R+4}\cdot d_T^2(\cdot, A) \Big)
\Big) \le 1.
\]
\end{proposition}

Finally, let us state the concentration result for polynomials in an abstract version.

\begin{proposition}\label{P:abstract-polynomials}
If a flip-swap random walk on $\calBn$ with some stationary distribution $\pi$ and a generator $L$ satisfies the stability condition~\eqref{eq:stability-condition}, then for any tetrahedral polynomial $f\colon\calBn\to\RR$ of degree $d$
\[
\pi\big(\big|f - \pi(f)\big| \ge t\big)
\le
2\exp\Big(
-\frac{1}{C_d}\min_{1\le r \le d} \min_{\mathcal{J}\in P_r}
\Big(
\frac{t}{R^{r/2}\|\pi( \nabla^r f)\|_{\mathcal{J}}}
\Big)^{2/|\mathcal{J}|}
\Big),
\]
where $C_d$ is a constant depending only on the degree $d$ of $f$.
\end{proposition}

\begin{remark}
Although Proposition~\ref{P:abstract_dalpha} gives a worse constant in the exponent than Theorem~\ref{T:d_alpha} even in the case of conditional Bernoulli distributions $\pi(p,k)$, we state it here as in principle it does not assume that $\pi$ satisfies the SCP and thus potentially can be applied in other settings.
\end{remark}

\begin{remark}
The above propositions can be transferred to more general random walks that change at each step at most a fixed number of coordinates $N$ (with $N=2$ in case of flip-swap random walks).
We do not pursue this direction though and do not write all the theorems in full generality for the sake of readability.
\end{remark}

\begin{remark}
Recently, Cryan et al.~\cite{MR4203344} have shown a version of Theorem~\ref{T:HS-main} for $k$-homogeneous strongly log-concave measures.
Strong log-concavity is yet another possible generalization of the SRP, which is in general incomparable with the SCP~\cite{MR4203344}.
It is known, cf. Br\"{a}nd\'{e}n and Huh~\cite{MR4172622}, that any $k$-homogeneous strongly log-concave measure is supported on the set of bases of some matroid of rank $k$.
Using this fact, and extending the previous results for uniform measures on the bases of matroids by Anari et al.~\cite{MR4003314} and Kaufman and Oppenheim~\cite{MR3857285}, Cryan et al.~\cite{MR4203344} explicitly construct a base-exchange random walk, which has any given strongly log-concave measure as a stationary distribution, and verify it satisfies the modified log-Sobolev inequality~\eqref{eq:mLSI}.

Since the base-exchange random walk proposed therein is a particular instance of a flip-swap random walk, a natural question is whether it satisfies the stability condition~\eqref{eq:stability-condition}, which would allow to deduce concentration results presented in this section.
Unfortunately, the answer seems to be negative in full generality as can be seen already in the case of independent Bernoulli random variables $B=(B_1,\ldots,B_n)$ with different probabilities of success $\PP(B_i=1)=p_i$ conditioned on their sum being $k$, i.e., for the distribution $\pi(p,k)\sim \mathcal{L}(B\,\vert\, \sum_i B_i=k)$.
If one chooses $p_1\to 1^-$ and $p_j=c$ for $j>1$ and some $c\in(0,1)$, then it is straightforward to verify that the base-exchange random walk of~\cite{MR4203344} is at best $k$-stable.
Therefore, applying propositions of Section~\ref{ss:abstract-formulations} to the base-exchange random walk gives much worse concentration constants than those of Section~\ref{S:CB_results}.
On the other hand, the flip-swap random walk proposed by Hermon and Salez~\cite{hermon2019modified} with stationary measure $\pi(p,k)$, which we use to prove the results of Section~\ref{S:CB_results}, turns out to be $2$-stable.

In view of the above, it is an interesting problem to analyze what other known kernels satisfy the stability condition~\eqref{eq:stability-condition} with good (dimension-independent) constant and to look for some other criteria that would allow to deduce this condition.
\end{remark}

\section{Proofs of the results of Section~\ref{S:general_SCP_results}}\label{S:general_SCP_proofs}
In this section we provide proofs of Theorems~\ref{T:d_alpha},~\ref{T:d_alpha_matrix} and~\ref{T:weak-martingale}.
All these results follow by modifications of the martingale argument due to Pemantle and Peres~\cite{MR3197973}.

Let $X\sim\pi$ be a random variable with values in $\calBn$ satisfying the SCP and
denote $\supp X = \{\, i\in\{1,\ldots,n\} : X_i = 1\,\}$.
In the non-homogeous case define a filtration $\calF=(\calF_l)_{l=0}^n$ by letting simply $\calF_0=\{\emptyset,\Omega\}$ and $\calF_l= \sigma(X_1,\ldots, X_l)$ for $l=1,\ldots,n$.  In the $k$-homogeous case introduce a family of random variables $Y_1,\ldots,Y_k$ given by the conditions
\begin{gather}\label{eq:definition-of-Y}
\mathcal{L}( Y_{1}\,\vert\,X ) = \Unif(\supp X\backslash \{1,\ldots,k\})
 \quad\text{and}\\
 \mathcal{L}(Y_{l}\,\vert\,X,Y_{1},\ldots,Y_{l-1}) =  \Unif(\supp X \setminus \{1,\ldots,k,Y_{1},\ldots, Y_{l-1}\}),
 \quad\text{for}\quad l=2,\ldots,k,\nonumber
\end{gather}
 where $\Unif(A)$ stands for the uniform distribution on the set $A$, and for notational simplicity we set $\Unif(\emptyset)$ to be the Dirac mass at $0$ and $X_0 \equiv 1$  (i.e., we add to $X$ an additional coordinate providing no information and if the above sampling scheme yields all elements from $\supp X$ before sampling some $Y_l$, we set $Y_i$ to zero for all $i\ge l$).
 Finally, define a filtration $\calG = (\calG_l)_{l=0}^{2k}$  setting $\calG_0=\{\emptyset,\Omega\}$ and $\calG_l=\sigma(X_1,\ldots,X_l)$ for $l\in [k]$, $\calG_{k+r}=\sigma(X_1,\ldots,X_k, Y_1,\ldots, Y_r\}$ for $r\in[k]$.

 In other words in the first $k$-steps the subsequent values of $X$ at the  first $k$ coordinates are revealed, while in the last $k$ steps one reveals in a uniformly random order the remaining coordinates at which $X$ takes the value $1$. Note that if $\alpha$ is nonincreasing (which we may assume without loss of generality) and $f$ is 1-Lipschitz with respect to $d_\alpha$ then the first part of this sampling scheme promotes the coordinates which may have the greatest impact on the value of $f(X)$. The construction can be thought of as a modification of the sampling scheme proposed by Pemantle and Peres in which one immediately starts revealing in a random order the coordinates at which $X$ takes the value $1$.

The proof of Theorems~\ref{T:d_alpha} and~\ref{T:d_alpha_matrix} will be based on the following two lemmas.
 \begin{lemma}\label{L:Ml_estimate} Let $\alpha \in \RR_+^n$ be nonincreasing and let $f\colon\calBn\to\calHd$ be $1$-Lipschitz with respect to the distance $d_\alpha$.
 Assume that $X$ is a $\calBn$-valued random vector satisfying
 the SCP. Let $M_l=\EE[f(X)\,\vert\,\calF_l]-\EE[f(X)\,\vert\,\calF_{l-1}]$ for $l\in [n]$. Then for every $l\in [n]$,
 \begin{equation}\label{eq:increment-upper-bound-1}
 M_l^2 \preccurlyeq 4\alpha_l^2 I_d.
 \end{equation}
 \end{lemma}

 \begin{lemma}\label{L:Ml_estimate-hom} In the setting of Lemma \ref{L:Ml_estimate} assume additionally that $X$ is $k$-homogeneous. For $l\in [2k]$ define
$N_l=\EE[f(X)\,\vert\,\calG_l]-\EE[f(X)\,\vert\,\calG_{l-1}]$. Then for $l \in  [k]$,
\begin{equation}\label{eq:increment-upper-bound-2}
N_l^2 \preccurlyeq 4\alpha_l^2 I_d,
\end{equation}
while for $l=k+1,\ldots,2k$,
 \begin{equation}\label{eq:increment-upper-bound-3}
 	N_l^2 \preccurlyeq
 	4\alpha_k^2 I_d.
 \end{equation}
 \end{lemma}

 We postpone for now the proof of the above lemmas and firstly show how they imply Theorems~\ref{T:d_alpha} and~\ref{T:d_alpha_matrix}.
To this end let us recall the matrix version  of the Azuma-Hoeffding inequality due to Tropp \cite[Theorem 7.1]{MR2946459}, which asserts that if $D_l$, $l=1,\ldots,n$ are $\calHd$-valued martingale differences and $D_l^2 \preccurlyeq C_l^2$ for some deterministic matrices $C_l\in \calHd$, then for all $t\ge 0$,
\begin{displaymath}
  \PP\Big(\lambda_{max}\Big(\sum_{l=1}^n D_l\Big)\ge t\Big) \le de^{-t^2/8\sigma^2},
\end{displaymath}
where $\sigma^2 = \|\sum_{l=1}^n C_l^2\|$.
Note also that for $d=1$ the classical Azuma-Hoeffding inequality (see, e.g., \cite[Theorem 5.8]{MR2547432}) allows to replace the constant $1/8$ by $1/2$.

 \begin{proof}[Proof of Theorems~\ref{T:d_alpha} and~\ref{T:d_alpha_matrix}]
 Since the SCP is invariant under permutations of coordinates of $X$, we may and do assume that $\alpha=\alpha^{\downarrow}$.
 By Lemma~\ref{L:Ml_estimate} the martingale differences $M_l$ satisfy $M_l^2 \preccurlyeq C_l^2 := 4\alpha_l^2 I_d$. Clearly
 \begin{equation}\label{eq:non-homogeneous}
	\Big\|\sum_{l=1}^n C_l^2 \Big\| = 4\abs{\alpha}^2.
 \end{equation}
 If $X$ is $k$-homogeneous, then by Lemma~\ref{L:Ml_estimate-hom}, $N_l^2 \preccurlyeq \widetilde{C}_l^2 := 4\alpha_{\min(l,k)}^2 I_d$. In this case
 \begin{equation}\label{eq:homogeneous}
	\Big\|\sum_{l=1}^{2k} \widetilde{C}_l^2\Big\| =
	4 \Big[
	\Big(\sum_{l=1}^k \alpha_l^2 \Big) + k\alpha_k^2
	\Big]
	\le
	8 \sum_{l=1}^k \alpha_l^2.
 \end{equation}

We have $f(X) = \sum_{l=1}^n M_l$, whereas in the $k$-homogeneous case $f(X) = \sum_{l=1}^{2k} N_l$ (observe that after $2k$-steps of the sampling procedure all the nonzero coordinates of $X$ are revealed and so $X$ is $\mathcal{G}_{2k}$-measurable). Thus the conclusion of Theorem~\ref{T:d_alpha} follows by applying estimates~\eqref{eq:non-homogeneous} and~\eqref{eq:homogeneous} for $d=1$ together with the classical Azuma-Hoeffding inequality. Similarly, Theorem~\ref{T:d_alpha_matrix} follows from the matrix version of the Azuma-Hoeffding inequality.
\end{proof}

It remains to prove Lemmas \ref{L:Ml_estimate} and \ref{L:Ml_estimate-hom}.

  \begin{proof}[Proof of Lemma~\ref{L:Ml_estimate}]
  Let $A^x_l=\{ X_1=x_1,\ldots,X_l=x_l \}$ for $x=(x_1,\ldots,x_n)\in\calBn$ and $l=0,\ldots,n$.
  Then, for $l = 1,\ldots,n$ and any $x\in\calBn$ such that $\PP(A_l^x) > 0$,
  \begin{multline*}
  	\EE[ f(X)\,\vert\, A^x_l ] -\EE [f(X)\,\vert\, A^x_{l-1}]
  	=
  	\EE[ f(X)\,\vert\, A^x_{l-1}, X_l=x_l] - \EE[f(X)\,\vert\, A^x_{l-1}]
  	\\=
  	\PP(X_l\neq x_l\,\vert\, A^x_{l-1})
  	\big(
  	\EE[f(X)\,\vert\, A^x_{l-1},X_l=x_l]
  	-
  	\EE[f(X)\,\vert\, A^x_{l-1},X_l\neq x_l]
  	\big).
  \end{multline*}
  If $\PP(X_l\neq x_l\,\vert\, A^x_{l-1})\neq 0$, then by the SCP there exist a coupling $(\hat{X},\hat{Y})$ between the distributions $\mathcal{L}(X\,\vert\,A^x_{l-1},X_l=x_l)$ and $\mathcal{L}(X\,\vert\,A^x_{l-1},X_l\neq x_l)$ that is supported on the set $\{\,(y,z)\in\mathcal{B}_n^2\,\colon\, d_H((y_i)_{i>l}, (z_i)_{i>l}) \le 1 \,\}$.
  Using this coupling, the Lipschitz property of $f$, Jensen's inequality  and the fact that $\alpha_i \le \alpha_l$ for any $i>l$, we get that
  \begin{multline*}
	\norm{\EE[ f(X)\,\vert\, A^x_l ] -\EE [f(X)\,\vert\, A^x_{l-1}]}
	\\\le
    \PP(X_l\neq x_l\,\vert\, A^x_{l-1}) \EE \| f((x_i)_{i\le l},\hat{X}_{i>l}) - f((x_i)_{i< l},1-x_l,\hat{Y}_{i>l})\| \\
	\le \PP(X_l\neq x_l\,\vert\, A^x_{l-1})
	\cdot
	2\alpha_l \le 2\alpha_l,
  \end{multline*}
  which is equivalent to \eqref{eq:increment-upper-bound-1}.
\end{proof}

\begin{proof}[Proof of Lemma \ref{L:Ml_estimate-hom}] Note that for $l\le k$, we have $\calG_l = \calF_l$. As a consequence $N_l = M_l$, where $M_l$ are martingale increments defined in Lemma \ref{L:Ml_estimate}, which implies \eqref{eq:increment-upper-bound-2}.

Consider now $l > k$ of the form $l=k+r$ and for $x = (x_1,\ldots,x_k) \in \mathcal{B}_k$ and $v = (v_1,\ldots, v_k) \in (\{0\}\cup\{k+1,\ldots,n\})^k$ set $A^{x,v}_l=\{ X_1=x_1,\ldots,X_k=x_k, Y_1 = v_1,\ldots,Y_r = v_r \}$.
Then $\calF_l$ is generated by the sets $A^{x,v}_l$. By the definition of the variables $Y_r$, we have $\{Y_r = i\} \subseteq \{X_i = 1\}$ and so for any $x,v$ such that $\PP(A^{x,v}_l)>0$,
\begin{align}\label{eq:conditional-probability-expanded}
  \EE [f(X)\,|\, A^{x,v}_l] = \frac{\EE [f(X)\ind{A^{x,v}_{l-1}}\indbr{X_{v_r} = 1}\indbr{Y_r=v_r}]}{\PP(A^{x,v}_{l-1},X_{v_r}=1, Y_r = v_r)}.
\end{align}
For $s \in [r]$ let $m_s = |\{i\in [k]\colon x_i= 1\}| + |\{j \in [s-1] \colon v_j \neq 0\}|$ be the number of ones sampled by the time $k+s-1$.
It follows from \eqref{eq:definition-of-Y} that if $m_s < k$ then  $\PP( A^{x,v}_{k+s}) > 0$ implies that $v_s \neq 0$ and $\PP(Y_s = v_s|X,Y_1,\ldots,Y_{s-1}) = \frac{1}{k-m_s}$ on $A^{x,v}_{k+s-1}\cap\{X_{v_s} = 1\}$,  whereas if $m_s=k$, then $\PP( A^{x,v}_{k+s}) > 0$ implies that $v_s = 0$ and $\PP(Y_s = v_s\,|\,X, Y_1,\ldots,Y_{s-1}) = 1$ on $A^{x,v}_{k+s-1}\cap\{X_{v_s} = 1\} = A^{x,v}_{k+s-1}$.
Going back to \eqref{eq:conditional-probability-expanded} and using this observation for $s=r,\ldots,1$, we obtain that
\begin{displaymath}
  \EE [f(X)\,|\, A^{x,v}_l] = \EE[ f(X)\,|\,B^{x,v}_l],
\end{displaymath}
where $B^{x,v}_l = \{X_1=x_1,\ldots,X_k = x_k, X_{v_1}=\ldots = X_{v_{l-k}} = 1\}$. We thus obtain
\begin{multline*}
  	\EE[ f(X)\,\vert\, A^{x,v}_l ] -\EE [f(X)\,\vert\, A^{x,v}_{l-1}] \\
  =  \PP(X_{v_r} \neq 1\,|\,B^{x,v}_{l-1})(\EE[ f(X)\,|\,B^{x,v}_{l-1}, X_{v_r} = 1] - \EE[ f(X)\,|\,B^{x,v}_{l-1}, X_{v_r} \neq 1]).
\end{multline*}
Note that the right-hand side may be non-zero only if $v_r\neq 0$. In this case
using the inequality $\alpha_{v_s} \le \alpha_k$ for $s \in [k]$ we can conclude as in the proof of Lemma~\ref{L:Ml_estimate}.

\end{proof}

Let us now pass to the proof of Theorem \ref{T:weak-martingale}.

\begin{proof}[Proof of Theorem \ref{T:weak-martingale}]

We will rely on the martingale used in the article by Pemantle and Peres.
Let $X$ be a random vector with law $\pi$ and define the random variables $Y_l$ for $l \le n$ as
\begin{gather}\label{eq:definition-of-Y-2}
\mathcal{L}( Y_{1}\,\vert\,X ) = \Unif(\supp X)
 \quad\text{and}\\
 \mathcal{L}(Y_{l}\,\vert\,X,Y_{1},\ldots,Y_{l-1}) =  \Unif(\supp X \setminus \{Y_{1},\ldots, Y_{l-1}\}),
 \quad\text{for}\quad l=2,\ldots,k,\nonumber
\end{gather}
i.e., $Y_l's$ reveal in a uniformly random order the elements of $\supp X$.
Let $\mathcal{H}_0=\{\emptyset,\Omega\}$ and $\mathcal{H}_l = \sigma(Y_1,\ldots,Y_l)$ for $l=1,\ldots,k$.
Then $f(X) - \EE f(X)= \sum_{l=1}^k \EE[f(X)\,|\,\mathcal{H}_l] - \EE[f(X)\,|\,\mathcal{H}_{l-1}] =: \sum_{l=1}^k D_l$. We will use the matrix version of Freedman's inequality due to Tropp \cite{MR2802042}, which asserts (in a version specialized for our application) that if $\|D_l\|\le a$ a.s. for all $l$, and $\|\sum_{l=1}^k  \EE[ D_l^2\,|\,\mathcal{H}_{l-1}]\| \le \sigma^2$ a.s., then for any $t \ge 0$,
\begin{equation}\label{eq:Freedman}
  \PP(\|f(X) - \EE f(X)\| \ge t) \le 2d\exp\Big(-\frac{t^2}{2 \sigma^2 + 2at/3}\Big).
\end{equation}

Consider thus a sequence of pairwise distinct $v_1,\ldots,v_k \in [n]$ and denote $A^v_l = \{Y_1 = v_1,\ldots,Y_l = v_l\}$. Similarly as in the proof of Lemma \ref{L:Ml_estimate-hom}, if $\PP(A^v_l) > 0$, then we have
\begin{displaymath}
  \EE[f(X)\,|\,A^v_l] = \EE[f(X)\,|\,B^v_l],
\end{displaymath}
where $B_l^v = \{X_{v_1}= \ldots=X_{v_l} = 1\}$. Therefore we have
\begin{equation}\label{eq:increment}
  D_l\ind{A^v_l} =
  \PP(X_{v_{l}}=0\,|\,B^v_{l-1})
  \big(
  \EE[f(X)\,|\, B^v_{l-1},X_{v_l}=1] -
  \EE[f(X)|B^v_{l-1},X_{v_l}=0]
  \big)\ind{A^v_l}.
\end{equation}

Since the SRP implies the SCP, there exists a coupling $(\tilde{Z},\hat{Z})$ between the distributions $\mathcal{L}(X\,|\,B^v_l)$ and $\mathcal{L}(X\,|\,B^v_{l-1},X_{v_l}=0)$ such that $\tilde{Z}$ and $\hat{Z}$ differ just at coordinate $v_l$ and one additional coordinate (at which by $k$-homogeneity $\hat{Z}$ necessarily takes the value one). Let $\tilde{Y}_l$ be this coordinate.
We have
\begin{align}\label{eq:increment-1}
\EE [f(X)\,|\, B^v_{l-1},X_{v_l}=1] -
\EE [f(X)\,|\, B^v_{l-1},X_{v_l}=0] =
\EE [f(\tilde{Z}) - f(\hat{Z})],
\end{align}

Since $\hat{Z}^{\tilde{Y}_l}=\tilde{Z}^{v_l}$, we have
\begin{multline}\label{eq:long-formula}
  \Big(
  \EE[f(X)\,|\, B^v_{l-1},X_{v_l}=1] -
  \EE[f(X)\,|\,B^v_{l-1},X_{v_l}=0]
  \Big)^2
  =
  \Big(
  \EE [f(\tilde{Z}) - f(\hat{Z})]
  \Big)^2 \\
  \preccurlyeq
  \EE \Big[\big(
  f(\tilde{Z}) - f(\hat{Z})\big)^2\Big]
   =
  \EE \Big[\big(
  f(\tilde{Z}) - f(\tilde{Z}^{v_l}) +
  f(\hat{Z}^{\tilde{Y}_l}) - f(\hat{Z})
  \big)^2\Big] \\
   \preccurlyeq
   2 \EE \Big[ \big(f(\tilde{Z}) - f(\tilde{Z}^{v_l})\big)^2\Big] +
   2 \EE \Big[\big(f(\hat{Z}^{\tilde{Y}_l}) - f(\hat{Z})\big)^2\Big]
   \preccurlyeq 2 C_{v_l}^2 + 2 \EE C_{\tilde{Y}_l}^2,
\end{multline}
where in the first and second inequality we used the operator convexity of the function $x\mapsto x^2$ (see \cite[Example V.1.3]{MR1477662}), and in the last inequality the assumption \eqref{eq:matrix-bounded difference-assumption}.

In particular, using \eqref{eq:increment}, we obtain $\|D_l^2\|\le 4\max_i\| C_i^2\|$, so $\|D_l\| \le  2 K$.
Moreover as on $A^v_l$ we have $Y_l = v_l$, by \eqref{eq:increment} and \eqref{eq:long-formula} we get that
\begin{align*}
  D_l^2\ind{A^v_l} \preccurlyeq  2( C_{Y_l}^2 + \EE C_{\tilde{Y}_l}^2)\PP(  X_{v_l}=0\,|\,B^v_{l-1})^2\ind{A^v_l}.
\end{align*}

Let us now slightly change our notation and think of $\tilde{Y}_l$ as of random variable defined on the same probability space as $X$, with conditional distribution with respect to the $\sigma$-field $\mathcal{H}_l$ given on each of its atoms $A^v_l$ by the above construction, using the corresponding coupling (which depends on $v_1,\ldots,v_l$). Then the above inequality can be written as
\begin{align}\label{eq:increment-squared}
  D_l^2 \preccurlyeq  2\sum_{v_l \in [n]\setminus \{ v_1,\ldots,v_{l-1}\}} \big( C_{Y_l}^2 + \EE [ C_{\tilde{Y}_l}^2\,|\, A^v_{l} ] \big)\PP(  X_{v_l}=0\,|\,B^v_{l-1})^2\ind{A^v_l}.
\end{align}
Let us now go back to the equations \eqref{eq:increment} and \eqref{eq:increment-1} and let us apply them in the special case of the function $\tilde{f}(x) = \sum_{i=1}^n x_i C_i^2$, denoting the corresponding martingale increment by $\tilde{D}_l$. We obtain that
\begin{displaymath}
  \tilde{D}_l\ind{A^v_l} =  \PP(X_{v_{l}}=0\,|\,B^v_{l-1}) \big(C^2_{Y_l} - \EE[C^2_{\tilde{Y}_l}\,|\,A^v_l]\big)\ind{A^v_l}.
\end{displaymath}
Thus we get that
\begin{displaymath}
  0 = \EE [\tilde{D}_l\,|\,A^v_{l-1}]
  =
  \sum_{v_l\in [n]\setminus \{v_1,\ldots,v_{l-1}\}}
  \EE \Big[\PP(X_{v_{l}}=0\,|\,B^v_{l-1})\ind{A^v_l}
  \big(
  C^2_{Y_l} - \EE[C^2_{\tilde{Y}_l}\,|\,A^v_l]\big)
  \,\Big|\,A^v_{l-1}\Big],
\end{displaymath}
i.e.,
\begin{multline*}
  \sum_{v_l\in [n]\setminus \{v_1,\ldots,v_{l-1}\}}
  \EE \Big[\PP(X_{v_{l}}=0\,|\,B^v_{l-1})\ind{A^v_l} C^2_{Y_l}\,\Big|\,A^v_{l-1}\Big]\\
  = \sum_{v_l\in [n]\setminus \{v_1,\ldots,v_{l-1}\}}
  \EE \Big[\PP(X_{v_{l}}=0\,|\,B^v_{l-1})\ind{A^v_l} \EE (C^2_{\tilde{Y}_l}\,|\,A^v_l)\,\Big|\,A^v_{l-1}\Big],
\end{multline*}
which combined with the estimate \eqref{eq:increment-squared} on $D_l^2$ (replacing $\PP(  X_{v_l}=0\,|\,B^v_{l-1})^2$ by $\PP(  X_{v_l}=0\,|\,B^v_{l-1})$) gives
\begin{align*}
\EE [D_l^2\,|\,A^v_{l-1}]
&\preccurlyeq
2
\sum_{v_l \in [n]\setminus \{ v_1,\ldots,v_{l-1}\}}
\EE \Big[ (C_{Y_l}^2 + \EE [C_{\tilde{Y}_l}^2\,|\,A^v_l]) \PP(  X_{v_l}=0\,|\,B^v_{l-1})\ind{A^v_l}\,\Big|\,A^v_{l-1}\Big] \\
& =
4 \sum_{v_l \in [n]\setminus \{ v_1,\ldots,v_{l-1}\}}
\EE\Big[ C_{Y_l}^2\PP(  X_{v_l}=0\,|\,B^v_{l-1})\ind{A^v_l}  \,\Big|\,A^v_{l-1}\Big]\\
& \preccurlyeq  4\sum_{v_l \in [n]\setminus \{v_1,\ldots,v_{l-1}\}} C_{v_l}^2  \PP(A^v_l\,|\,A^v_{l-1}) \\
& = 4\sum_{v_l \in [n]\setminus \{v_1,\ldots,v_{l-1}\}} C_{v_l}^2 \frac{1}{k-l+1} \PP(X_{v_l} = 1\,|\,B^v_{l-1})\\
& \preccurlyeq 4\sum_{v_l \in [n]\setminus \{v_1,\ldots,v_{l-1}\}} C_{v_l}^2 \frac{1}{k-l+1} \PP(X_{v_l} = 1),
\end{align*}
where in the last inequality we used \cite[Lemma 1.10]{MR3899605}, which asserts that $\PP(X_{v_l} = 1) \ge \PP(X_{v_l} = 1|B^v_{l-1})$ (we remark that this is the only place in the proof in which we use the full strength of the strong Rayleigh property).

Extending the summation to $[n]$, we thus obtain
\begin{displaymath}
  \EE [D_l^2\,|\,\mathcal{H}_{l-1}] \preccurlyeq 4\sum_{v=1}^n C_v^2 \PP(X_v=1) \frac{1}{k-l+1},
\end{displaymath}
whence
\begin{displaymath}
  \sum_{l=1}^{k} \EE[D_l^2|\mathcal{H}_{l-1}]
  \preccurlyeq 4 \sum_{v=1}^n C_v^2 \PP(X_v=1) \log (ek)
  \preccurlyeq 4\log (ek) \cdot \EE\Big[ \sum_{v=1}^n X_v C_v^2\Big].
\end{displaymath}
Combining this with the already obtained bound on $\|D_l\|$ allows us to apply \eqref{eq:Freedman} with $a = 2K$ and $\sigma^2 = 4\|\EE \sum_{v=1}^n X_v C_v^2\|\log (ek)$, which ends the proof of the theorem.
\end{proof}

\section{Proofs of the results of Section~\ref{S:abstract_formulations}}\label{S:abstract_proofs}

\subsection{Propositions~\ref{P:abstract_dalpha} and~\ref{P:abstract_matrix}}
The main idea behind the proof of Proposition~\ref{P:abstract_dalpha} is to find an estimate on $\norm{\Gamma_+(f)}_\infty$ in terms of $\alpha$, refining~\eqref{eq:vf-dh-estimate}, and then to use the Herbst argument.
We will need the following lemma which we state in the matrix setting as it will be useful for the proof of Proposition~\ref{P:abstract_matrix} as well.

\begin{lemma}\label{L:analytical_mtx}
Let $t=(t_1,\ldots,t_n)$ be a sequence of nonnegative numbers and let $D_1,\ldots,D_n\in\calHd$ be positive semidefinite matrices.
Then for any $T_1\ge \abs{t}_1$ and $T_\infty\ge \abs{t}_\infty$
\begin{align}\label{eq:supremum-KM}
	\big\|
	\sum_{i=1}^n t_iD_i
	\big\|
	\le
	T_\infty\cdot
	\sup\Big\{	\,
	\big\|	
	\sum_{i\in\mathcal{I}} D_i
	\big\|	
	\colon
	\mathcal{I}\subset [n],\,
	\abs{\mathcal{I}}\le \ceil{T_1 / T_\infty}
	\,\Big\}.
\end{align}

\end{lemma}

\begin{proof}
By homogeneity we may assume without loss of generality that $T_\infty = 1$. We may also assume that $T_1$ is a positive integer. Let
\begin{displaymath}
  \mathcal{X} = \Big\{\,x \in [0,1]^n\colon \sum_{i=1}^n x_i \le T_1\,\Big\}, \; \mathcal{Y} = \Big\{\,y \in \{0,1\}^n\colon \sum_{i=1}^n y_i \le T_1\,\Big\}.
\end{displaymath}
Since the right-hand side of \eqref{eq:supremum-KM} equals to $\max\{\|\sum_{i=1}^n y_i D_i\|\colon y \in \mathcal{Y}\}$, whereas the left-hand side is a convex function of $t$, the lemma will follow once we prove that $\mathcal{X} \subset \conv \mathcal{Y}$. To this end, by the Krein--Milman theorem, it is enough to show that $\mathcal{Y}$ is the set of all extreme points of the closed convex set $\mathcal{X}$. Consider any $x \in \mathcal{X}\setminus \mathcal{Y}$.
Let $i_0 \in [n]$ be such that $x_{i_0} \in (0,1)$.
If $\sum_i x_i < T_1$ then for $\varepsilon$ sufficiently close to zero, $x+\varepsilon e_{i_0}, x - \varepsilon e_{i_0} \in \mathcal{X}$ and so $x = \frac{1}{2}(x+\varepsilon e_{i_0}) + \frac{1}{2}(x-\varepsilon e_{i_0})$ is not an extreme point of $\mathcal{X}$.
If  $\sum_i x_i = T_1$, then since $T_1$ is an integer, there exists ${i_1}\neq {i_0}$ such that $x_{i_1} \in (0,1)$.
Then $x = \frac{1}{2}u + \frac{1}{2}v$, where $u = x + \varepsilon e_{i_0} - \varepsilon e_{i_1}$, $v = x - \varepsilon e_{i_0} + \varepsilon e_{i_1}$.
For $\varepsilon$ close to zero $u,v \in \mathcal{X}$, thus again, $x$ is not an extreme point.
\end{proof}

\begin{proof}[Proof of Proposition~\ref{P:abstract_dalpha}]
We recall that for $x\in\calBn$ and $i,j\in [n]$, $x^{i}$ and $x^{ij}$ denote the vectors obtained from $x$ by flipping the $i$-th and swapping the $i$-th and $j$-th coordinates respectively.
For any $x\in\calBn$, using the definition~\eqref{eq:Gamma+def} of $\Gamma_+$, Lipschitz property of $f$ and inequality $(a+b)^2\le 2(a^2+b^2)$ we get
\begin{align}
\begin{split}\label{eq:another-long-formula}
 \Gamma_+(f)(x)
 &=
 \sum_{i=1}^{n}
 (f(x)-f(x^{i}))_+^2 L(x,x^{i})
 +
 \frac{1}{2}
 \sum_{i,j=1}^{n}
 (f(x)-f(x^{ij}))_+^2 L(x,x^{ij})
 \\
 &\le
 \sum_{i=1}^{n} \alpha_i^2L(x,x^{i})
 +
 \frac{1}{2}
 \sum_{i,j=1}^{n} (\alpha_i+\alpha_j)^2L(x,x^{ij})\indbr{x\neq x^{ij}} \\
 &\le
 \sum_{i=1}^{n} \alpha_i^2L(x,x^{i})
 +
 2\sum_{i=1}^{n}\alpha_i^2\sum_{j=1}^{n}L(x,x^{ij})\indbr{x\neq x^{ij}}
 \le
 2\sum_{i=1}^{n}\alpha_i^2\sum_{y\colon y_i\neq x_i} L(x,y).
 \end{split}
\end{align}
Therefore, by the stability condition~\eqref{eq:stability-condition} we estimate $\norm{\Gamma_+(f)}_{\infty} \le 2R\rho(L) \abs{\alpha}^2$.
Herbst's argument~\eqref{eq:Herbsst_arg} allows to conclude the first part.

The second part of the proposition follows by observing that for a flip-swap random walk
\[
 \sum_{i=1}^n \sum_{y\colon y_i\neq x_i}L(x,y) \le 2\cdot \Delta(L)
\]
so by~\eqref{eq:another-long-formula}, Lemma~\ref{L:analytical_mtx}  applied in the scalar setting $d=1$ with $t_i = 2\sum_{y\colon y_i\neq x_i}L(x,y)$, $D_i=\alpha_i^2$, $T_1=4\Delta(L)$ and $T_{\infty}=4R\rho(L)$ we can estimate
\[
 \norm{\Gamma_+(f)}_{\infty}
 \le
 4R\rho(L) \sum^{\ceil{\Delta(L)/R\rho(L)}}_{i=1}(\alpha_i^\downarrow)^{2}
\]
and conclude again in virtue of Herbst's argument~\eqref{eq:Herbsst_arg}.
\end{proof}

The proof of Proposition~\ref{P:abstract_matrix} follows along similar lines to the proof of Proposition~\ref{P:abstract_dalpha}, the difference being that in the end, instead of Herbst's argument, we apply the concentration result of Aoun et al.~\cite{MR4108222}, which  asserts that if $L$ satisfies the matrix Poincar\'{e} inequality with constant $C_P>0$
\begin{equation}\label{eq:Poincare}
 \operatorname{Var}(f)
 \preccurlyeq
 - C_P \pi( f Lf)
 \quad
 \forall\, f\colon\calBn\to\calHd,
\end{equation}
(where $L$ acts on the matrix-valued function $f$ elementwise and $fLf$ is the matrix product), then it satisfies the exponential concentration bound of the form
\begin{equation}\label{eq:matrix_poincare_concentration}
\pi\big(
\lambda_{max}(f-\pi(f)) > t
\big)
\le
d\exp\Big(
\frac{-t^2}{2C_P v_f + t\sqrt{2C_P v_f}}
\Big),
\end{equation}
where $v_f=\sup_x\norm{\Gamma(f)(x)}$ (where $\Gamma$ is defined via \eqref{eq:Gamma_def}, again with matrix multiplication, and $\|\cdot\|$ stands for the operator norm).
Note that for $d=1$, \eqref{eq:Poincare} is just the usual scalar Poincar\'e inequality.

\begin{proof}[Proof of Proposition~\ref{P:abstract_matrix}]
	For any $x\in\calBn$ and $i,j\in[n]$, using operator convexity of the function $x\mapsto x^2$ (see \cite[Example V.1.3]{MR1477662}) we get that
	\begin{multline}\label{eq:operator_conv_triangle}
	\big( f(x) - f(x^{ij}) \big)^2
	=
	\big[ \big(f(x)-f(x^i)\big) + \big(f(x^i)- f(x^{ij})\big) \big]^2
	\\
	\preccurlyeq
	2\big(f(x)-f(x^i)\big)^2+
	2\big(f(x^i)-f(x^{ij})\big)^2.
	\end{multline}
	Therefore, by the definition~\eqref{eq:Gamma_def} of $\Gamma$, by the assumed Lipschitz property~\eqref{eq:abstract_matrix_f_condition} of $f$ and by~\eqref{eq:operator_conv_triangle}, for any $x\in\calBn$,
	\begin{align}\label{eq:abstract_matrix}
	\begin{split}
	 \Gamma (f)(x)
	 &=
	 \frac{1}{2}\sum_{i=1}^{n}
 	 (f(x)-f(x^{i}))^2 L(x,x^{i})
 	 +
 	 \frac{1}{4}
 	 \sum_{i,j=1}^{n}
 	 (f(x)-f(x^{ij}))^2 L(x,x^{ij})	
 	 \\
	 &\preccurlyeq
	 \frac{1}{2}
	 \sum_{i=1}^{n} C_i^2L(x,x^i)
	 +
	 \frac{1}{2}
     \sum_{i,j=1}^{n} (C_i^2+C_j^2)L(x,x^{ij})\indbr{x\neq x^{ij}}
     \\
     &\preccurlyeq
     \sum_{i=1}^{n} C_i^2\cdot
	 \big[ \sum_{y\colon y_i\neq x_i}L(x,y) \big].
	 \end{split}
	\end{align}
	As both hand sides of~\eqref{eq:abstract_matrix} are positive semidefinite, their norms compare as well.
	Therefore, as in the proof of Proposition~\ref{P:abstract_dalpha}, by Lemma~\ref{L:analytical_mtx} with $t_i=\sum_{y\colon y_i\neq x_i}L(x,y)$, $T_1 = 2\Delta(L)$, $T_{\infty} = 2R\rho(L)$ and $D_i=C_i^2$
	\[
	\sup_{x\in\calBn}
	\norm{\Gamma(f)(x)}
	\le 	
	2R\rho(L)\cdot
	\sup\Big\{	
	\big\|	
	\sum_{i\in\mathcal{I}} C_i^2
	\big\|	
	\colon
	\mathcal{I}\subset [n],\,
	\abs{\mathcal{I}}\le \ceil{\Delta(L) / R\rho(L)}
	\Big\}.
	\]
	Since $L$ satisfies the (scalar) modified log-Sobolev inequality~\eqref{eq:mLSI}, then it satisfies the (scalar) Poincar\'{e} inequality with constant $C_P=2/\rho(L)$ (see, e.g., \cite[p. 292]{MR2283379}, noting slightly different definitions of constants in functional inequalities used therein) and whence by~\cite[Proposition 2.2]{MR4216521} or~\cite[Theorem 1.1]{MR4245747} it satisfies the matrix Poincar\'{e} inequality \eqref{eq:Poincare} with the same constant, which yields the conclusion in virtue of~\eqref{eq:matrix_poincare_concentration}.
\end{proof}
\subsection{Proposition~\ref{P:abstract_Talagrand}}
The proof of Proposition~\ref{P:abstract_Talagrand} is based on the idea introduced by Boucheron et al.~\cite{MR2540852} and then developed by Paulin~\cite{MR3248197}.
We follow the exposition from the works of Sambale and Sinulis~\cite{sambale2019modified,sambale2020concentration}.
We start with the following lemmas.
\begin{lemma}\label{L:Gamma_+bound}
For any flip-swap random walk with generator $L$ satisfying the stability condition~\eqref{eq:stability-condition} and for any $A\subset\calBn$
\begin{equation}\label{eq:L_Gamma_+bound1}
	\Gamma_+(d_T^2(\cdot, A))(x) \le 8R\rho(L)\cdot d_T^2(x,A).
\end{equation}
Moreover, for any $x,y\in\calBn$ and any set $A\subset \calBn$,
\begin{equation}\label{eq:L_Gamma_+bound2}
	d_T^2(x,A) - d_T^2(y,A) \le d_H(x,y).
\end{equation}
\end{lemma}
\begin{proof}
For $x\in\calBn$, $\alpha \in \RR^n$ and a probability measure $\mu$ on $\calBn$, let $h_x(\mu,\alpha) = \sum_{i}\alpha_i \mu(z\colon z_i\neq x_i)$.
By Sion's minmax theorem, cf.~\cite[p. 227]{MR3185193},
\begin{equation}\label{eq:minmax_dt}
	d_T(x,A) = \inf_{\mu\in\mathcal{M}(A)}\sup_{\alpha\in B_2^n} h_x(\mu,\alpha),
\end{equation}
where $\mathcal{M}(A)$ is the set of probability measures on $A$ and $B_2^n =\{\, x\in\RR^n\colon \abs{x}\le 1\,\}$ is the unit ball in $\RR^n$.
Let $\alpha^\ast\in \RR^n_+\cap B_2^n$, $\mu^\ast\in\mathcal{M}(A)$ be such that $d_T(x,A) = h_x(\mu^\ast,\alpha^\ast)$ and set $\nu_y=\argmin_{\nu\in\mathcal{M}(A)} h_y(\nu, \alpha^\ast)$.
Then
\begin{align*}
 \Gamma_+\big(d_T(\cdot,A)\big)(x)
 &=
 \sum_{y} \big[
 h_x(\mu^\ast, \alpha^\ast) - \inf_{\nu\in\mathcal{M}(A)}\sup_{\alpha\in B_2^n} h_y(\nu,\alpha)
 \big]_+^2L(x,y)\\
 &\le
 \sum_{y} \big[
 h_x(\mu^\ast, \alpha^\ast) - h_y(\nu_y,\alpha^\ast)
 \big]_+^2L(x,y)\\
 &\le
 \sum_{y} \big[
 h_x(\nu_y, \alpha^\ast) - h_y(\nu_y,\alpha^\ast)
 \big]_+^2L(x,y)\\
 &=
 \sum_y
 \big[
 \sum_i
 \alpha^\ast_i
 \big(
 \nu_y(z\colon z_i\neq x_i)-\nu_y(z\colon z_i\neq y_i)
 \big)
 \big]_+^2L(x,y)\\
 &\le
 \sum_y\big[
 \sum_i \alpha^\ast_i
 \indbr{x_i\neq y_i}
 \big]^2L(x,y)\\
 &\le
 2 \sum_i (\alpha^\ast_i)^2 \sum_{y\colon y_i \neq x_i} L(x,y)
 \le
 2R\rho(L),
\end{align*}
where the penultimate inequality follows since $L$ is a flip-swap random walk and therefore $L(x,y) > 0$ implies that $d_H(x,y)\le 2$ and so at most two elements of the sum $\sum_i \alpha_i^\ast\indbr{x_i\neq y_i}$ are non-zero, whence we may apply the inequality $(a+b)^2\le 2(a^2+b^2)$. The last inequality is a consequence of the condition $\alpha^\ast \in B_2^n$ and the stability condition~\eqref{eq:stability-condition}.
We conclude~\eqref{eq:L_Gamma_+bound1} using the definition of $\Gamma_+$ and estimating $(a-b)^2_+(a+b)_+^2\le 4a^2(a-b)_+^2$.

To show the second part, note that~\eqref{eq:minmax_dt} together with the Cauchy--Schwarz inequality imply that
\[
	d_T^2(x,A) = \inf_{\mu\in\mathcal{M}(A)}
	\sum_i
	\big(\mu(z\colon z_i\neq x_i) \big)^2
	=
	\sum_i
	\big(\mu^\ast_x(z\colon z_i\neq x_i) \big)^2
\]
for some $\mu^\ast_x\in\mathcal{M}(A)$.
Therefore, for any $x,y\in\calBn$
\begin{equation*}
	d_T^2(x,A) - d_T^2(y,A)
	\le
	\sum_i
	\Big[
	\big(\mu^\ast_x(z\colon z_i\neq x_i) \big)^2
	-
	\big(\mu^\ast_x(z\colon z_i\neq y_i) \big)^2
	\Big]
	\le
	\sum_{i} \indbr{x_i\neq y_i},
\end{equation*}
as desired.
\end{proof}

Using the inequality $1-e^{-z}\le z$ we observe that for any $f\colon\calBn\to\RR$,
\begin{align*}
 \calE(e^f,f)
 =
 \sum_{x}\pi(x)e^{f(x)}
 \Big[
 \sum_y
  (f(x)-f(y))_+(1- e^{f(y)-f(x)})L(x,y)
 \Big]
 \le
 \pi\big( e^{f}\Gamma_+(f) \big).
\end{align*}
Therefore, the modified log-Sobolev inequality~\eqref{eq:mLSI} implies the following inequality stated in Bobkov and G\"{o}tze~\cite{MR1682772}:
\begin{equation}\label{eq:mLSI-Bobkov}
	\rho(L)\Ent_\pi(e^f) \le \pi\big( e^f\tilde{\Gamma}(f)^2 \big)
\end{equation}
with operator $\tilde{\Gamma}(f) = \sqrt{\Gamma_+(f)}$ (note that in~\cite{MR1682772} $\tilde{\Gamma}$ is denoted by $\Gamma$, we use $\tilde{\Gamma}$ to avoid a conflict of notation).
As a consequence, the hypothesis of~\cite[Theorem 2.1]{MR1682772} (formula (1.1)) therein holds under the assumption of the modified log-Sobolev inequality~\eqref{eq:mLSI} (with $c=2/\rho(L)$).
As a result, the following lemma follows directly by the derivation of~\cite[equation~(2.4)]{MR1682772} with a slight adjustment of constants (see also \cite{MR1305064}).
\begin{lemma}\label{L:Bobkov-Gotze}
If a measure $\pi$ on $\calBn$ satisfies the modified log-Sobolev inequality~\eqref{eq:mLSI} and $f\colon\calBn \to [0,\infty)$ is such that $\Gamma_+(f)\le Cf$ for some constant $C>0$, then for all $t>C/\rho(L)$,
\begin{equation}\label{eq:mLSI=>mLSI}
	\pi\big(\exp(f/t)\big)
	\le
	\exp\Big(
	\frac{\pi(f)}{t-C/\rho(L)}
	\Big).
\end{equation}
\end{lemma}

We are finally in position to prove Proposition~\ref{P:abstract_Talagrand}.	
\begin{proof}[Proof of Proposition~\ref{P:abstract_Talagrand}]
To lighten notation, denote $f(x) = d_T^2(x,A)$ for $x \in \calBn$ and some fixed set $A\subset\calBn$.
Denote also $h(z) = ({e^z-1})/{z}$ for $z \in [0,\infty)$ and $Df_y(x)=f(x)-f(y)$ for $x,y\in\calBn$, and note that $h$ is an increasing function. Starting with the modified log-Sobolev inequality~\eqref{eq:mLSI} we have for all $\lambda>0$,
\begin{align*}
 \Ent_\pi(e^{-\lambda f})
 &\le \lambda/\rho(L)\cdot  \calE(e^{-\lambda f}, -f)	\\
 &=
 \lambda /\rho(L)
 \sum_{x,y}
 \big( Df_y(x) \big)_+
 \big(
 e^{-\lambda f(y)}-e^{-\lambda f(x)}
 \big)
 L(x,y)\pi(x)
 \tag{by reversibility of $L$}
 \\
 &=
 \lambda^2 /\rho(L) \sum_x \pi(x) e^{-\lambda f(x)}
 \Big[
 \sum_y
 \big( Df_y(x) \big)_+^2
 h\big( \lambda Df_y(x) \big)
 L(x,y)\Big]
 \\
 \tag{by~\eqref{eq:L_Gamma_+bound2}}
 &\le
 \lambda^2h(2\lambda)/\rho(L)\cdot\pi\big( e^{-\lambda f}\Gamma_+(f) \big) \\
 &\le
 8R\lambda^2h(2\lambda)\cdot \pi\big(e^{-\lambda f} f\big)
 \tag{by~\eqref{eq:L_Gamma_+bound1}}\\
 &\le
 8R\lambda^2h(2\lambda)\cdot \pi(e^{-\lambda f})\pi(f),
\end{align*}
where the last inequality follows by convexity of the function $t\mapsto t\log t$.
Therefore, using the entropy method (cf., e.g.,~\cite[Chapter 6]{MR3185193}) and monotonicity of $h$, we have for every $\lambda>0$,
\begin{align*}
 \pi\big(
 \exp(\lambda(\pi(f)-f)
 \big)
 &=
 \exp\Big(
 \lambda\int_0^\lambda\frac{d}{ds}\Big[
 \frac{1}{s}\log\pi(e^{-sf})
 \Big]\,ds
 \Big)
 \\
 &=
 \exp\Big(
 \lambda\int_0^\lambda
 \frac{\Ent_\pi(e^{-sf)}}{s^2\pi(e^{-sf})}\,ds
 \Big)
 \\
 &\le
 \exp\Big(\lambda\cdot
 8R\pi(f)\int_0^\lambda
 h(2s)\,ds
 \Big)
 \le
 \exp\big(
 4R\lambda (e^{2\lambda}-1) \pi(f)
 \big).
\end{align*}
By Chebyshev's exponential inequality
\begin{equation*}
 \pi(A) =
 \pi\big(\pi(f)-f \ge \pi(f)\big)
 \le
 \exp\Big(\lambda
 \big(4R (e^{2\lambda}-1)-1\big)\pi(f)
 \Big).
\end{equation*}
Taking $\lambda = \frac{1}{2}\log(1+\frac{1}{8R})$ and estimating $\log(1+x)\ge x/(x+1)$ for $x\ge 1$ gives
\begin{equation}\label{eq:abstract-talagrand-proof-final}
	\pi(A)
	\le
	\exp\Big(
	-\frac{1}{4}
	\log\big(1+\frac{1}{8R}\big)\pi(f)
	\Big)
	\le
	\exp\Big(
	-\frac{\pi(f)}{32R+4}\Big).
\end{equation}
We conclude by dividing~\eqref{eq:abstract-talagrand-proof-final} by its right hand side  and  using Lemma~\ref{L:Bobkov-Gotze} with $t=4+40R$ and $C=8R\rho(L)$ (in virtue of Lemma~\ref{L:Gamma_+bound}).
\end{proof}

\subsection{Proposition~\ref{P:abstract-polynomials}}
Before we move to the proof of Proposition~\ref{P:abstract-polynomials}, let us comment a bit on a background result.
Using the equivalence between the modified log-Sobolev inequality~\eqref{eq:mLSI} and the family of Beckner inequalities together with the approach developed by Boucheron et al.~\cite{MR2123200}, it was shown in~\cite[Proposition~3.1]{adamczak2020modified} that the following moment estimate is implied by the modified log-Sobolev inequality. Below we will denote by $\|\cdot\|_p$ the norm in $L^p(\pi)$.

\begin{proposition}\label{P:moment-estimate}
If a probability measure $\pi$ on $\calBn$ satisfies the modified log-Sobolev inequality~\eqref{eq:mLSI}, then for any $p\ge 2$,
\begin{equation}\label{eq:polys-moment-estimate}
 \norm{(f-\pi(f))_+}_{p}
 \le
 C
 \sqrt{p/{\rho(L)}}
 \norm{\sqrt{\Gamma_+(f)}}_{p},
\end{equation}
where $C=\sqrt{{3\sqrt{e}}/({\sqrt{e}-1})}$.
\end{proposition}

A general method of deriving estimates for polynomials from moment inequalities of the form \eqref{eq:polys-moment-estimate} has been presented in~\cite{MR3383337} in the continuous case, and in~\cite{MR3949267,adamczak2020modified} in the context of Glauber dynamics.
To obtain results for flip-swap random walks we will adapt a version of this method introduced recently by Sambale and Sinulis~\cite{sambale2020concentration} for multislices.

\begin{proof}[Proof of Proposition~\ref{P:abstract-polynomials}]
Below we write $C$ to denote universal constants and $C_a$ to denote constants depending only on the parameter $a$. In both cases the constants may change values between occurrences.
Let $f\colon\calBn\to\RR$ be a tetrahedral polynomial.
By $\partial_i$ we denote the partial derivative with respect to the $i$-th coordinate.
If $x,y\in \calBn$ differ at the $i$-th coordinate only, then by the fact that $f$ is linear in each coordinate
\[
 |f(x)-f(y)| =
 \big|\partial_i f(x) \big|.
\]

Similarly, if $x$ and $y$ differ only by a swap of the $i$-th and $j$-th coordinate, we have

\begin{multline*}
  |f(x) - f(y)| = |\partial_i f(x)(y_i-x_i) + \partial_j f(x)(y_j-x_j) + \partial_i\partial_j f(x)(y_i-x_i)(y_j-x_j)| \\
  \le |\partial_i f(x)|+|\partial_j f(x)| + |\partial_i\partial_j f(x)|.
\end{multline*}
Thus 
\begin{multline*}
  \Gamma(f)(x) = \frac{1}{2}\sum_{i=1}^n (f(x) - f(x^i))^2L(x,x^i) + \frac{1}{2}\sum_{1 \le i < j \le n} (f(x) - f(x^{ij}))^2 L(x,x^{ij})\\
\le \frac{1}{2} \sum_{i=1}^n |\partial_i f(x)|^2 L(x,x^i) + \frac{3}{2}\sum_{\stackrel{1\le i < j \le n}{x^{ij}\neq x}} (|\partial_i f(x)|^2 + |\partial_j f(x)|^2 + |\partial_i\partial_j f(x)|^2)L(x,x^{ij})\\
\le R\rho(L) \Big(3.5 \sum_{i=1}^n |\partial_i f(x)|^2 + 0.75 \sum_{i,j=1}^n |\partial_i \partial_j f(x)|^2\Big),
\end{multline*}
where in the last inequality we used the stability condition \eqref{eq:stability-condition}. Note that since $f$ is tetrahedral, $\partial_i\partial_i f(x) = 0$ for all $i$.

Combining the above equality with Proposition \ref{P:moment-estimate} we obtain that for every tetrahedral polynomial $f\colon \calBn\to \RR$
\begin{align}\label{eq:second-order-moment}
  \|f - \pi(f)\|_p \le C\sqrt{p}\sqrt{R} \Big(\big\| |\nabla f|\big\|_p  + \big\| \|\nabla^2 f\|_{HS}\big\|_p\Big),
\end{align}
where $C$ is a universal constant.

In the subsequent part of the proof we are going to need some auxiliary notation. For $d$-tensors $A = (a_{\ii})_{\ii\in[n]^d}$, $B = (b_\ii)_{\ii\in[n]^d}$ define
\begin{displaymath}
  \langle A, B \rangle = \sum_{\ii\in [n]^d} a_\ii b_\ii.
\end{displaymath}

Let us now consider a family of stochastically independent random tensors $\{G^I\colon I \subseteq \NN, |I| \in \{1,2\}\}$, given by $G^{\{m\}} = (g^{\{m\}}_i)_{i\in[n]}$, $G^{\{l,k\}} = (g^{\{l,k\}}_{i,j})_{i,j\in[n]}$, with coefficients being i.i.d. standard Gaussian variables. Denote by $P_{d,\le2}$ the family of all partitions of the set $[d]$ into non-empty subsets of cardinality at most 2.
Finally, for any positive integers $d$ and $l$ and $\mathcal{J} = \{J_1,\ldots,J_l\} \in P_{d,\le 2}$ define the random $d$-tensor $G_\mathcal{J} = (\prod_{j=1}^l g^{J_j}_{\ii_{J_j}})_{\ii\in [n]^d}$.
For instance $G_{\{\{1,3\},\{2\}\}} = (g^{\{1,3\}}_{i_1i_3} g^{\{2\}}_{i_2})_{i_1,i_2,i_3 \in [n]}$.

Using the fact that the $p$-th moment of a mean zero Gaussian variable with variance $\sigma^2$ is for $p\ge 2$ comparable to $\sqrt{p}\sigma$ up to universal constants, we can rewrite \eqref{eq:second-order-moment} as

\begin{align}\label{eq:second-order-Gaussian}
  \|f(X) - \EE f(X)\|_p \le C\sqrt{R} \Big(\| \langle \nabla f(X), G^{\{1\}} \rangle \|_p + \|\langle \nabla^2 f(X), G^{\{1,2\}} \rangle \|_p\Big),
\end{align}
where $X$ is a random vector with law $\pi$, independent of the family $\{G^I\}$.

The inequality \eqref{eq:second-order-Gaussian} constitutes a basis for the induction argument leading to the following inequality valid for any $f\colon \calBn\to \RR$ , $d
\ge 1$ and $p\ge 2$,
\begin{align}\label{eq:chaos-bound}
\begin{split}
\|f(X) - \EE f(X)\|_p \le&  C_d \Big(\sum_{l=d}^{2d}\sum_{\mathcal{J}\in P_{l,\le 2}}R^{|\mathcal{J}|/2}\| \langle \nabla^l f(X), G_{\mathcal{J}}\rangle\|_p\\
&+\sum_{l=1}^{2d-2}\sum_{\mathcal{J}\in P_{l,\le 2}}R^{|\mathcal{J}|/2}\|\langle \EE_X \nabla^l f(X),G_\mathcal{J}\rangle\|_p\Big).
\end{split}
\end{align}

Before we prove the above estimate, let us show how it implies the statement of the proposition. If $f$ is a tetrahedral polynomial of degree $d$, then $\nabla^l f = 0$ for $l > d$, moreover $\nabla^d f$ is constant and so $\nabla^d f(X) = \EE \nabla^d f(X)$. Thus \eqref{eq:chaos-bound} reduces to
\begin{displaymath}
  \|f(X) - \EE f(X)\|_p \le C_d\sum_{l=1}^{d}\sum_{\mathcal{J}\in P_{l,\le 2}}R^{|\mathcal{J}|/2}\|\langle \EE_X \nabla^l f(X),G_\mathcal{J}\rangle\|_p.
\end{displaymath}

We can now use moment estimates for tetrahedral homogeneous polynomials in i.i.d. standard Gaussian variables due to Lata{\l}a \cite{MR2294983}, which assert that for any $l$-tensor $A = (a_\ii)_{\ii\in [n]^l}$ and $p\ge 2$,
\begin{displaymath}
  \|\langle A, G_{\{\{1\},\ldots,\{l\}\}}\rangle\|_p \le C_l \sum_{\mathcal{J}\in P_l} p^{|\mathcal{J}|/2}\|A\|_{\mathcal{J}}.
\end{displaymath}

Applying this inequality to $\langle \EE_X \nabla^l f(X),G_\mathcal{J}\rangle$ (we treat here $\EE_X\nabla^l f(X)$ as a $|\mathcal{J}|$-tensor by merging the indices according to the partition $\mathcal{J}$), we obtain
\begin{displaymath}
  \|f(X) - \EE f(X)\|_p \le C_d\sum_{l=1}^{d}\sum_{\mathcal{J}\in P_{l,\le 2}}R^{|\mathcal{J}|/2}\sum_{\mathcal{I} \in P_l\colon \mathcal{I} \succ \mathcal{J}}p^{|\mathcal{I}|/2}\|\EE \nabla^l f(X)\|_{\mathcal{I}},
\end{displaymath}
where $\mathcal{I} \succ \mathcal{J}$ if every element of $\mathcal{I}$ is a union of certain elements of $\mathcal{J}$.
Rearranging the terms and taking into account that in a non-trivial case $R$ is bounded away from zero by an absolute constant (see Remark \ref{rem:R-lower-bound}), which gives $R^{|\mathcal{J}|/2} \le C_d R^{l/2}$ for $\mathcal{J} \in P_{l,\le 2}$, we get
\begin{displaymath}
  \|f(X) - \EE f(X)\|_p \le C_d \sum_{l=1}^{d} \sum_{\mathcal{I} \in P_l} R^{l/2} p^{|\mathcal{I}|/2}\|\EE\nabla^l f(X)\|_\mathcal{I}
\end{displaymath}
for $p \ge 2$. This implies the tail inequality of the proposition in the standard way by the use of Chebyshev's inequality $\PP(|f(X) - \EE f(X)| \ge e\|f(X)-\EE f(X)\|_p) \le e^{-p}$ followed by an appropriate change of variables and adjustment of constants. We leave the details to the reader and turn to the proof of \eqref{eq:chaos-bound}.

We will proceed by induction on $d$. For $d=1$, using the definitions of $G_{\{1\}}$ and $G_{\{\{1,2\}\}}$ one can easily see that \eqref{eq:chaos-bound} reads as
\begin{align*}
\|f(X) - \EE f(X)\|_p \le&  C\Big(\sqrt{R} \| \langle \nabla f(X), G^{\{1\}} \rangle \|_p + \sqrt{R} \|\langle \nabla^2 f(X), G^{\{1,2\}} \rangle \|_p \\
&+ R\|\langle \nabla^2 f(X), G_{\{\{1\},\{2\}\}} \rangle \|_p\Big),
\end{align*}
which is clearly weaker than \eqref{eq:second-order-Gaussian}. Let us thus assume that the inequality holds for all positive integers smaller than $d$. Applying the inequality with $d-1$ and combining it with the triangle inequality in $L_p$  we get (recall that the value of $C_d$ may change between occurences)
\begin{align}\label{eq:almost-the-end}
\begin{split}
\|f(X) - \EE f(X)\|_p \le&  C_d \Big(\sum_{l=d-1}^{2d-2}\sum_{\mathcal{J}\in P_{l,\le 2}}R^{|\mathcal{J}|/2}\| \langle \nabla^l f(X), G_{\mathcal{J}}\rangle\|_p\\
&+\sum_{l=1}^{2d-4}\sum_{\mathcal{J}\in P_{l,\le 2}}R^{|\mathcal{J}|/2}\|\langle \EE_X \nabla^l f(X),G_\mathcal{J}\rangle\|_p\Big) \\
\le & C_d\Big(\sum_{l=d-1}^{2d-2}\sum_{\mathcal{J}\in P_{l,\le 2}}R^{|\mathcal{J}|/2}\| \langle \nabla^l f(X), G_{\mathcal{J}}\rangle - \langle \EE_X \nabla^l f(X), G_{\mathcal{J}}\rangle\|_p\\
&+\sum_{l=1}^{2d-2}\sum_{\mathcal{J}\in P_{l,\le 2}}2R^{|\mathcal{J}|/2}\|\langle \EE_X \nabla^l f(X),G_\mathcal{J}\rangle\|_p\Big).
\end{split}
\end{align}

An application of inequality \eqref{eq:second-order-Gaussian} conditionally on $G_\mathcal{J}$ to the functions $g_{l,\mathcal{J}}(x) = \langle \nabla^l f(x),G_\mathcal{J}\rangle$ for $l =d-1,\ldots,2d-2$ and $\mathcal{J}\in P_{l,\le 2}$ (note that $g_{l,\mathcal{J}}$'s are tetrahedral polynomials), followed by the Fubini theorem, gives
\begin{multline*}
  \| \langle \nabla^l f(X), G_{\mathcal{J}}\rangle - \langle \EE_X \nabla^l f(X), G_{\mathcal{J}}\rangle\|_p \le C\sqrt{R}\Big(\|\langle \nabla^{l+1} f(X),G_{\mathcal{J}\cup\{\{l+1\}\}}\rangle\|_p \\+
  \|\langle \nabla^{l+2} f(X), G_{\mathcal{J}\cup\{\{l+1,l+2\}\}}\rangle\|_p\Big),
\end{multline*}
which combined with \eqref{eq:almost-the-end} concludes the induction step, thus proving \eqref{eq:chaos-bound}.
\end{proof}

\section{Proofs of the results of Section~\ref{S:CB_results}}\label{S:CB_proofs}
By virtue of the abstract results of Section~\ref{S:abstract_formulations}, all the results of Section~\ref{S:CB_results} follow from the observation that there exist a flip-swap random walk on $\calBn$ with stationary measure $\pi=\pi(p,k)$ that satisfy the stability condition~\eqref{eq:stability-condition} with constant $R=2$ for all $p\in(0,1)^n$ and $k=0,\ldots, n$ (cf.~Proposition~\ref{P:L-stability}).
The rest of this section is devoted to proving this fact.

We start with the following lemma, the proof of which is postponed until the end of this section.
\begin{lemma}\label{L:CB_coupling}
For every $n\in\NN$, $p\in (0,1)^n$ and $k\in [n]$, there exist a coupling $(Z,Z')$ of measures $\pi(p,k)$ and $\pi(p,k-1)$ such that
for all $x \in \supp \pi(p,k-1)$, and $r \in [n]$ such that $x_r = 0$,
\begin{equation}\label{eq:CB_coupling}
	\PP(Z=x+e_r\,\vert\, Z'=x)
	=
	\EE
	\Big[	
	\frac{\indbr{Z_r=1}}{\sum_{l=1}^n\indbr{Z_l=1}\indbr{x_l=0}}
	\Big]
\end{equation}
and for all $x \in \supp \pi(p,k)$, and $r\in [n]$ such that $x_r = 1$,
\begin{equation}\label{eq:CB_coupling2}
	\PP(Z'=x-e_r\,\vert\, Z=x)
	=
	\EE
	\Big[	
	\frac{\indbr{Z'_r=0}}{\sum_{l=1}^n\indbr{Z'_l=0}\indbr{x_l=1}}
	\Big].
\end{equation}
\end{lemma}

Our approach to defining an $R$-stable generator $L_\pi$ will be based on the inductive construction of Hermon and Salez~\cite{hermon2019modified} in the $k$-homogeneous case, which we will now recall. The construction works for any $k$-homogeneous probability measure $\pi$ on $\calBn$, satisfying the SCP and produces a generator of a $\pi$-reversible flip-swap random walk $Q^\ast$ such that $\rho(Q^\ast) \ge 1$ and $\Delta(Q^\ast) \le 2k$.

To simplify the notation, we are going to treat vectors $x_{\neq  l}$ for $x \in \calBn$ and $l \in [n]$ sometimes as elements of $\{0,1\}^{[n]\setminus\{l\}}$ (this is how they were defined at the beginning of Section \ref{S:general_SCP_results}) and sometimes as elements of $\mathcal{B}_{n-1}$ (with the natural identification, i.e., preserving the order of coordinates). The exact meaning will be clear from the context.  The same convention will apply to random vectors, e.g., to $X_{\neq l}$.

In the case $n=1$, we let $Q$ be the trivial generator on a one-point space. Clearly then $\rho(Q) = \infty$ and $\Delta(Q) = 0$.

For $n>1$, $l\in [n]$ and $x,y\in \supp\pi$, $x\neq y$, we set
\begin{align}\label{def:Ql}
	Q^{(l)}(x,y) =
	\begin{cases}
		\Prob{U=y_{\neq l}\,\vert\, V=x_{\neq l}}
		\Prob{X_l \neq x_l}	
		\quad&\text{if}\quad
		x_l\neq y_l,\\
		Q^{(l)}_{x_l}(x_{\neq l}, y_{\neq l})
		\quad&\text{else},
	\end{cases}
\end{align}
where $X$ is a random vector with law $\pi$ and $(U,V)$ is any coupling between $\mathcal{L}(X_{\neq l}\,\vert\, X_l = y_l)$ and $\mathcal{L}(X_{\neq l}\,\vert\, X_{l} = x_{l})$ given by the SCP\footnote{We note a small typo in the cited preprint version of~\cite{hermon2019modified} -- for $Q$ to be self-adjoint we need to have $\pi(x)$ in the denominator of each expression in~\cite[equation (90)]{hermon2019modified}. This change is also consistent with the subsequent part of the proof in \cite{hermon2019modified}.} and $Q^{(l)}_{x_l}$ is any flip-swap generator on $\mathcal{B}_{n-1}$ with stationary distribution $\mathcal{L}(X_{\neq l}\,\vert\, X_{l} = x_{l})$ such that $\rho(Q^{(l)}_{x_l})\ge 1$ and $\Delta(Q^{(l)}_{x_l})\le 2(k-x_i)$, the existence of which is provided by the induction scheme. We define the diagonal elements of $Q^{(l)}$ so that the row sums vanish.
Finally put
\begin{align}\label{eq:Qstar}
	Q^{\ast} = \frac{1}{n}\sum_{l=1}^n Q^{(l)}.
\end{align}
Then by (the proof of)~\cite[Theorem 2]{hermon2019modified}, we have $\rho(Q^\ast)\ge 1$, $\Delta(Q^\ast)\le 2k$.

Now we are in position to construct the generator $L_{\pi}$.
Let $X\sim \pi = \pi(p,k)$ for some $p\in (0,1)^n$ and $k\in\{0,\ldots,n\}$.
Observe that for any $y_l \in \{0,1\}$, we have $\mathcal{L}(X_{\neq l}\,\vert\, X_l = y_l) = \pi(p_{\neq l}, k-y_l)$, in particular in the above recursive construction we can restrict our attention to the class of conditional Bernoulli distributions and use as $Q_{x_l}^{(l)}$ the generators defined for such measures in dimension $n-1$. Moreover, for $(U,V)$ we can take the coupling $(Z,Z')$ (if $y_l = 0$) or $(Z',Z)$ (if $y_l = 1$) given by Lemma \ref{L:CB_coupling} applied in dimension $n-1$ with $p_{\neq l}$ instead of $p$ (note that since the right-hand side of \eqref{eq:CB_coupling} summed over $r$ such that $x_r = 0$ gives one, we indeed have $Z \triangleright Z'$, which makes this coupling a legitimate choice in the Hermon--Salez construction). Let us define $L_\pi$ as the outcome of the Hermon--Salez construction with the above choices of $Q_{x_l}^{(l)}$ and $(U,V)$. Thus, formally for $n = 1$ we let $L_\pi$ be the trivial generator and for $n > 1$ and $l \in [n]$ we set
\begin{align}\label{eq:kernel-construction}
	L_\pi = \frac{1}{n}\sum_{l=1}^n L^{(l)}
\end{align}
with
\begin{align}\label{def:Ll}
	L^{(l)}(x,y) =
	\begin{cases}
		\Prob{U=y_{\neq l}\,\vert\, V=x_{\neq l}}
		\Prob{X_l \neq x_l}	
		\quad&\text{if}\quad
		x_l\neq y_l,\\
		L_{\pi_l}(x_{\neq l}, y_{\neq l})
		\quad&\text{else},
	\end{cases}
\end{align}
for $x\neq y$, where $(U,V)$ is the coupling of $\pi(p_{\neq l},k- y_l)$ and $\pi(p_{\neq l},k - x_l)$ given by Lemma \ref{L:CB_coupling}, and $\pi_l = \pi(p_{\neq l},k-y_l)$ (again the diagonal elements are adjusted so that the row sums vanish).

Then, the results by Hermon and Salez, specialized to $L_\pi$ give

\begin{proposition}\label{P:HS-kernel-equiv}
The generator $L_\pi$ constructed according to~\eqref{eq:kernel-construction} generates a reversible flip-swap random walk with stationary measure~$\pi$ such that $\rho(L_{\pi})\ge 1$ and $\Delta(L_{\pi})\le 2k$.
\end{proposition}

Our main result concerning conditional Bernoulli distributions, underlying all the results from Section \ref{S:CB_results} is

\begin{proposition}\label{P:L-stability}
The generator $L_\pi$ constructed according to~\eqref{eq:kernel-construction} with stationary measure~$\pi$ satisfies the stability condition~\eqref{eq:stability-condition} with $R=2$.
\end{proposition}

\begin{proof}[Proof of Proposition~\ref{P:L-stability}]

We proceed by induction in the dimension $n$.

For $n=1$  the only possibilities are $k=0$ and $k=1$ and in both cases the left-hand side of \eqref{eq:stability-condition} vanishes. Thus the stability condition~\eqref{eq:stability-condition} is satisfied with any nonnegative $R$.

Assume the induction hypothesis holds for $n-1$ and fix $x\in\supp\pi$ and $i\in [n]$.
We may and do assume that $k \in \{1,\ldots,n-1\}$ as otherwise $L_\pi$ trivializes.

Since $\rho(L_\pi) \ge 1$, it is enough to show that
\begin{align}\label{eq:stability-simplified}
  \max_{x\in\supp\pi;\,i\in [n]} \sum_{y\colon y_i\neq x_i} L_\pi(x,y) \le 2.
\end{align}

As in the definition of $L_\pi$ we will denote by $X$ a random variable with distribution $\pi$.
\medskip

If $x_i=0$, then by \eqref{eq:kernel-construction}
\begin{multline}\label{eq:CaseI_start}
	\sum_{y\colon y_i\neq x_i} L_\pi(x,y) = \sum_{j\colon x_j = 1} \frac{1}{n}\sum_{l=1}^n L^{(l)}(x,x^{ij})\\
    = \sum_{l\in [n]\setminus\{i,j\}} \frac{1}{n}\sum_{j\colon x_j = 1} L^{(l)}(x,x^{ij}) + \frac{1}{n}\sum_{j\colon x_j = 1} L^{(i)}(x,x^{ij}) + \frac{1}{n}\sum_{j\colon x_j = 1} L^{(j)}(x,x^{ij}),
\end{multline}
	where we recall that $x^{ij}=x+e_i-e_j$.
	We estimate each term on the right hand side separately.

For $l \in [n]$ let $\zeta_l$ be the unique increasing bijection between $[n]\setminus\{l\}$ and $[n-1]$. 
If $l \neq i,j$, then for $y = x^{ij}$ we have $y_l = x_l$ and so, by \eqref{def:Ll}, $L^{(l)}(x,y) = L_{\pi_l}(x_{\neq l},y_{\neq l})$, where $\pi_l = \pi(p_{\neq l},k-x_l)$. Thus, denoting $r_l = \zeta_l(i)$, we get
\begin{multline}\label{eq:CaseI_I}
  \sum_{l\in [n]\setminus\{i,j\}} \frac{1}{n}\sum_{j\colon x_j = 1} L^{(l)}(x,x^{ij}) = \frac{1}{n} \sum_{l\in [n]\setminus\{i\}}\sum_{j\neq l\colon x_j = 1} L_{\pi_l}(x_{\neq l},(x^{ij})_{\neq l}) \\ = \frac{1}{n}\sum_{l\in [n]\setminus\{i\}}\sum_{y\in \mathcal{B}_{n-1}\colon y_{r_l} \neq (x_{\neq l})_{r_l}} L_{\pi_l}(x_{\neq l},y) \le \frac{n-1}{n}\cdot 2,
\end{multline}
where the last inequality follows from the induction assumption applied to $\pi_l$.

	The second term of~\eqref{eq:CaseI_start} is estimated again using the definition~\eqref{def:Ll}. Indeed, if $x_j = 1$, then for $y = x^{ij}$ we have $x_i \neq y_i$. Thus, recalling that $(U,V)$ is a coupling between the laws $\mathcal{L}(X_{\neq i}\,\vert\, X_i =1)$ and $\mathcal{L}(X_{\neq i} \,\vert\, X_i = 0)$ such that $V \triangleright U$, we obtain
	\begin{multline}\label{eq:CaseI_II}
	\frac{1}{n}\sum_{j\colon x_j = 1} L^{(i)}(x,x^{ij})
	=
	\frac{1}{n}
	\sum_{j\colon x_j=1}
	\Prob{U = (x^{ij})_{\neq i} \,\vert\, V = x_{\neq i}}
 \Prob{X_i \neq x_i}\\
	=
	\frac{1}{n}
	\PP\big(
		X_i=1	
	\big)		
	\le
	\frac{1}{n}.
	\end{multline}
	
Let us pass to the last term of~\eqref{eq:CaseI_start}. We stress that this is the crucial part of the proof, the only one in which we use the specific form of the coupling $(U, V)$ used in the construction of $L_\pi$.

To estimate this last term we use~\eqref{eq:CB_coupling} from Lemma~\ref{L:CB_coupling} combined with the fact that if $x_i = 0$ and $x_j = 1$, then for $y=x^{ij}$, $y_j = 0 \neq x_j$ and so $(U, V)$ from \eqref{def:Ll} is the coupling between the laws $\pi(p_{\neq j}, k)$ and $\pi(p_{\neq j}, k-1)$ given by Lemma \ref{L:CB_coupling} (in dimension $n-1$).
For $j \in [n]$ consider a $\mathcal{B}_{n-1}$-valued random vector $Z^{(j)}\sim\mathcal{L}(X_{\neq j}\,\vert\,X_j=0)=\pi(p_{\neq j}, k)$.
	Note also that since $X,x$ have the same number of ones, we have
\begin{align}\label{eq:zmiana}
\sum_{l=1}^n\indbr{X_l=0}\indbr{x_l=1}=\sum_{l=1}^n\indbr{X_l=1}\indbr{x_l=0}.
\end{align}

	Putting all the above observations together, we obtain
	\begin{align}\label{eq:CaseI_III}
	\notag
	\frac{1}{n}
	\sum_{j\colon x_j=1}
	L^{(j)}(x,x^{ij})
	&\stackrel{\text{\eqmakebox[a][c]{{\rm Lemma}\, \ref{L:CB_coupling},\, \eqref{def:Ll}}}}{=}
	\frac{1}{n}
	\sum_{j\colon x_j=1}
	\EE\big[
		\frac{\indbr{Z^{(j)}_i=1}}{\sum_{l\neq j} \indbr{Z^{(j)}_l=1}\indbr{x_l=0}}
	\big]
	\PP\big( X_j=0\big)\\
	\begin{split}
    &\stackrel{\text{\eqmakebox[a][c]{}}}{=}
	\frac{1}{n}
	\sum_{j\colon x_j=1}
	\EE\big[
		\indbr{X_i=1}
		\frac{\indbr{X_j=0}}{\sum_{l\neq j} \indbr{X_l=1}\indbr{x_l=0}}
	\big]\\
    &\stackrel{\text{\eqmakebox[a][c]{$x_j=1$}}}{=}
	\frac{1}{n}
	\sum_{j\colon x_j=1}
	\EE\big[
		\indbr{X_i=1}
		\frac{\indbr{X_j=0}}{\sum_{l} \indbr{X_l=1}\indbr{x_l=0}}
	\big]\\
	&\stackrel{\text{\eqmakebox[a][c]{\eqref{eq:zmiana}}}}{=}
	\frac{1}{n}
	\sum_{j\colon x_j=1}
	\EE\big[
		\indbr{X_i=1}
		\frac{\indbr{X_j=0}}{\sum_l \indbr{X_l=0}\indbr{x_l=1}}
	\big]\\
    &\stackrel{\text{\eqmakebox[a][c]{}}}{=}
	\frac{1}{n}
	\PP\big(
		X_i=1	
	\big)\le \frac{1}{n}.
	\end{split}
	\end{align}
	
Combining the estimates~\eqref{eq:CaseI_I}, \eqref{eq:CaseI_II} and \eqref{eq:CaseI_III} with~\eqref{eq:CaseI_start} yields \eqref{eq:stability-simplified} and thus the stability condition~\eqref{eq:stability-condition} with $R=2$ in the case $x_i=0$.
The case $x_i=1$ is analogous, the main difference being that in~\eqref{eq:CaseI_III} we use \eqref{eq:CB_coupling2} in place of~\eqref{eq:CB_coupling} from Lemma~\ref{L:CB_coupling}.

Together the two cases give the induction step and conclude the proof of the proposition.
\end{proof}

Let us conclude this section with the proof of Lemma~\ref{L:CB_coupling}.

\begin{proof}[Proof of Lemma~\ref{L:CB_coupling}]
For $x\in\calBn$, let $\kappa(x) = \sum_i x_i$ and let $B$ be a vector of independent Bernoulli random variables with probabilities of success given by $p=(p_1,\ldots,p_n)$.
Consider three $\calBn$-valued random variables: $\widehat{Z}\sim\mathcal{L}(B\,\vert\,\kappa(B)=k)$, $Z'\sim\mathcal{L}(B\,\vert\,\kappa(B)=k-1)$ and $Z$ such that for all $x,y \in \calBn$,
\begin{align}\label{eq:conditional-probability-ZZ}
  \PP(Z = y|Z' = x) = h(y,x),
\end{align}
where
\begin{equation*}
	h(y,x) =
	\EE
	\Big[	
	\frac{\indbr{\widehat{Z}_r=1}}{\sum_{l}\indbr{\widehat{Z}_l=1}\indbr{x_l=0}}
	\Big]
\end{equation*}
if $y=x+e_r$ for some $r\in[n]$ and $\kappa(x)=k-1$, and  $h(y,x)=0$ otherwise.
Note that for $x\in\calBn$ such that $\kappa(x)=k-1$,  $\sum_{l}\indbr{\widehat{Z}_l=1}\indbr{x_l=0} > 0$ with probability one, so $h(y,x)$ is well defined.
Moreover, for such $x$
\[
	\sum_{y\in \calBn} h(y,x) =
	\sum_{r\colon x_r = 0} h(x+e_r,x) = 1,
\]
which guarantees the existence of the couple $(Z,Z')$ satisfying \eqref{eq:conditional-probability-ZZ}. Thus to prove \eqref{eq:CB_coupling} it is enough to show that $Z \sim \widehat{Z}$, i.e., that $\sum_{x \in \calBn}  h(y,x)\PP(Z'=x)=\PP(\widehat{Z}=y)$ for any $y\in\calBn$ such that $\kappa(y)=k$.

Observe that for any $r\in [n]$ such that $x_r=0$ and $\kappa(x) = k-1$
\begin{equation}\label{eq:CB_coupling_quotient}
	\frac{\Prob{Z'=x}}{\PP({\widehat{Z}=x+e_r})}
	=
	\frac{\Prob{B=x}}{\Prob{B=x+e_r}}
	\frac{\Prob{\kappa(B)=k}}{\Prob{\kappa(B)=k-1}}
	=
	\frac{1-p_r}{p_r}
	\frac{\Prob{\kappa(B)=k}}{\Prob{\kappa(B)=k-1}}.
\end{equation}
Moreover, for any $f\colon\calBn\to\RR$ and $r\in[n]$,
\begin{equation}\label{eq:CB_measure_change}
	\EE \big[f(B)\indbr{B_r=1}\big] \frac{1-p_r}{p_r}
	=
	\EE \big[f(B+e_r)\indbr{B_r=0}\big].
\end{equation}

We use~\eqref{eq:CB_coupling_quotient} and~\eqref{eq:CB_measure_change} to get that for any such $y$ and any $r\in[n]$ such that $y_r=1$ and $\kappa(y)=k$,
\begin{align}\label{eq:CB_equivalence}
 	\notag
	\frac{h(y,y-e_r)\PP(Z'=y-e_r)}{\PP(\widehat{Z}=y)}
	&\stackrel{\eqref{eq:CB_coupling_quotient}}{=}
	h(y,y-e_r)
	\frac{1-p_r}{p_r}
	\frac{\Prob{\kappa(B)=k}}{\Prob{\kappa(B)=k-1}}\\
\begin{split}
	&=
	\EE
	\Big[	
	\frac{\indbr{B_r=1}\indbr{\kappa(B)=k}}{\indbr{B_r=1}+\sum_{l\neq r}\indbr{B_l=1}\indbr{y_l=0}}
	\Big]	
	\frac{(1-p_r)/p_r}{\Prob{\kappa(B)=k-1}}\\
	&\stackrel{\eqref{eq:CB_measure_change}}{=}
	\EE
	\Big[	
	\frac{\indbr{B_r=0}\indbr{\kappa(B)=k-1}}{\indbr{B_r=0}+\sum_{l\neq r}\indbr{B_l=1}\indbr{y_l=0}}
	\Big]	
	\frac{1}{\Prob{\kappa(B)=k-1}}\\
	&=
	\EE
	\Big[	
	\frac{\indbr{B_r=0}\indbr{\kappa(B)=k-1}}{\indbr{B_r=0}+\sum_{l\neq r}\indbr{B_l=0}\indbr{y_l=1}}
	\Big]	
	\frac{1}{\Prob{\kappa(B)=k-1}}\\
	&=
	\EE
	\Big[	
	\frac{\indbr{Z'_r=0}}{\sum_{l}\indbr{Z'_l=0}\indbr{y_l=1}}
	\Big]
\end{split}	
\end{align}
where the penultimate step comes from the fact that for any $u,v$ such that $\kappa(u)=\kappa(v)$ one has $\sum \indbr{u=0}\indbr{v=1}=\sum \indbr{u=1}\indbr{v=0}$ applied to $u=\xi_r(B)$, $v=\xi_r(y)$, where $\xi_r$ is the projection from $\calBn$ to $\mathcal{B}_{n-1}$ obtained by skipping the $r$-th coordinate (note that if $B_r=0$ and $\kappa(B)=k-1$ then $\kappa(u)=\kappa(v)=k-1$).
Therefore by~\eqref{eq:CB_equivalence}
\[
	\frac{\sum_x h(y,x)\PP(Z'=x)}{\PP(\widehat{Z}=y)}
	=
	\frac{\sum_{r\colon y_r = 1} h(y,y-e_r)\PP(Z'=y-e_r)}{\PP(\widehat{Z}=y)}
	=1,
\]
which completes the proof of~\eqref{eq:CB_coupling}.
The equality~\eqref{eq:CB_coupling2} follows again by~\eqref{eq:CB_equivalence}.
\end{proof}

\bibliographystyle{amsplain}
\bibliography{SCP-concentration}
\end{document}